\documentclass{amsart}
\usepackage{amscd,amsmath,amssymb,amsfonts}
\usepackage[cmtip, all]{xy}

\usepackage{pinlabel}
\usepackage{graphicx}

\title{The Adams-Novikov spectral sequence and Voevodsky's slice tower}

\author{Marc Levine}
\address{Universit\"at Duisburg-Essen\\
Fakult\"at Mathematik, Campus Essen\\
45117 Essen\\
Germany}
\email{marc.levine@uni-due.de}
\urladdr{http://www.esaga.uni-due.de/marc.levine/}

\thanks{Research supported by the Alexander von Humboldt Foundation}



\newtheorem{thm}{Theorem}[section]
\newtheorem{prop}[thm]{Proposition}
\newtheorem{lem}[thm]{Lemma}

\renewcommand{\theclaim}{\kern-3pt}

\newtheorem{IntroThm}{Theorem}

\theoremstyle{definition}
\newtheorem{Def}[thm]{Definition}

\theoremstyle{remark}
\newtheorem{rem}[thm]{Remark}
\newtheorem{rems}[thm]{Remarks}
\newtheorem{ex}[thm]{Example}
\newtheorem{exs}[thm]{Examples}

\newtheorem{IntroRem}{Remark}

\numberwithin{equation}{section}

\newcommand{\sA}{{\mathcal A}}
\newcommand{\sB}{{\mathcal B}}
\newcommand{\sC}{{\mathcal C}}

\newcommand{\sE}{{\mathcal E}}

\newcommand{\sH}{{\mathcal H}}

\newcommand{\sM}{{\mathcal M}}
\newcommand{\sN}{{\mathcal N}}

\newcommand{\sS}{{\mathcal S}}
\newcommand{\sT}{{\mathcal T}}

\newcommand{\sX}{{\mathcal X}}
\newcommand{\sY}{{\mathcal Y}}
\newcommand{\sZ}{{\mathcal Z}}
\newcommand{\A}{{\mathbb A}}

\newcommand{\CC}{{\mathbb C}}

\newcommand{\G}{{\mathbb G}}

\newcommand{\N}{{\mathbb N}}
\renewcommand{\P}{{\mathbb P}}

\newcommand{\mS}{{\mathbb S}}

\newcommand{\ZZ}{{\mathbb Z}}

\renewcommand{\phi}{\varphi}

\renewcommand{\1}{{\mathbf{1}}}

\newcommand{\an}{{\rm an}}

\newcommand{\Hom}{{\rm Hom}}

\newcommand{\Ext}{{\rm Ext}}
\newcommand{\im}{{\rm im}}
\newcommand{\Spec}{\operatorname{Spec}}

\newcommand{\0}{\emptyset}
\newcommand{\sHom}{{\mathcal{H}{om}}}

\newcommand{\id}{{\operatorname{id}}}

\newcommand{\sk}{{\operatorname{\rm sk}}}

\newcommand{\holim}{\mathop{{\rm holim}}}
\newcommand{\op}{{\text{\rm op}}}

\newcommand{\del}{\partial}

\renewcommand{\max}{{\operatorname{\rm max}}}

\newcommand{\Spt}{{\mathbf{Spt}}}
\newcommand{\Spc}{{\mathbf{Spc}}}
\newcommand{\Sm}{{\mathbf{Sm}}}

\renewcommand{\lim}{\operatornamewithlimits{\varprojlim}}
\newcommand{\colim}{\operatornamewithlimits{\varinjlim}}

\newcommand{\Ho}{{\mathbf{Ho}}}

\newcommand{\sq}{\square}
 
\newcommand{\Tot}{{\operatorname{\rm Tot}}}

\newcommand{\SH}{{\operatorname{\sS\sH}}}

\newcommand{\eff}{{\mathop{eff}}}

\newcommand{\DM}{{DM}}

\newcommand{\sSets}{{\mathbf{sSets}}}

\newcommand{\ds}{{/\kern-3pt/}}

\newcommand{\hofib}{{\mathop{\rm{hofib}}}}
\newcommand{\hocofib}{{\mathop{\rm{hocofib}}}}

\newcommand{\EM}{{{EM}_{\A^1}}}

\newcommand{\mot}{{\mathop{mot}}}

\newcommand{\et}{{\operatorname{\acute{e}t}}}

\renewcommand{\:}{\kern-1.5pt:\kern-1.5pt}

\newcommand{\gr}{{\text{Gr}}}

\renewcommand{\Re}{{\mathop{Re}}}

\newcommand{\MGL}{\text{MGL}}

\newcommand{\MU}{{\text{MU}}}
\newcommand{\BP}{{\text{BP}}}

\newcommand{\Dec}{\text{Dec}}

\newcommand{\pt}{pt}

\newcommand{\An}{\text{An}}
\newcommand{\Map}{\text{Map}}

\begin{document}

\renewcommand{\abstractname}{Abstract}
\begin{abstract}  We show that the spectral sequence induced by the Betti realization of the slice tower for the motivic sphere spectrum  agrees with the Adams-Novikov spectral sequence, after a suitable reindexing. The proof relies on a partial extension of Deligne's d\'ecalage construction to the Tot-tower of a cosimplicial spectrum.
 \end{abstract}
\date{\today}
\maketitle
\tableofcontents

\section{Introduction} Voevodsky has defined a natural tower in the motivic stable homotopy category $\SH(k)$ over a field $k$, called the {\em slice tower} (see \cite{VoevICM, VoevSlice}). Relying on the computation of the slices of $\MGL$ by Hopkins-Morel \cite{HopkinsMorel},  complete proofs of which have been recently made available through the work of Hoyois \cite{Hoyois}, we have filled in the details of a proof of the conjecture of Voevodsky \cite{VoevSlice}, identifying the slices of the motivic sphere spectrum with a motive built out of the $E_1$-complex in the classical Adams-Novikov spectral sequence for the stable homotopy groups of spheres (see \cite{Adams}), explicitly:
\[
s_q(\mS_k)\cong \EM(N\pi_{2q}(\MU^{\wedge *+1})\otimes\ZZ(q)[2q]),
\]
where $n\mapsto \MU^{\wedge n+1}$ is the usual cosimplicial spectrum associated to the $E_\infty$ ring spectrum $\MU$, $N\pi_{2q}(\MU^{\wedge *+1})$ is the associated normalized complex of homotopy groups, $\ZZ(q)[2q]$ is the shifted Tate motive and $\EM$ is the Eilenberg-MacLane functor from Voevodsky's category of motives to the motivic stable homotopy category. 

 In addition,  the Betti realization of the slice tower yields a tower over the classical sphere spectrum $\mS$, and we have shown that the resulting spectral sequence is bounded and converges to the homotopy groups of $\mS$. Furthermore, we have also shown that the resulting comparison map from the homotopy sheaves $\pi_{n,0}$ of the slice tower, evaluated on any algebraically closed subfield of $\CC$,
 to the homotopy groups of the Betti realization is an isomorphism. For all these results, we refer the reader to \cite{LevineComp}.

 Putting all this together, we have a bounded spectral sequence,  converging to $\pi_*\mS$, of ``motivic origin" and whose $E_2$-term agrees with the $E_2$-term in the Adams-Novikov spectral sequence, after a reindexing. The question thus arises: are these two spectral sequences the same, again after reindexing? The main result of this paper is an affirmative answer to this question, more precisely:
 
 \begin{IntroThm} \label{thm:Main} Let $\MU\in \SH$ be the complex cobordism spectrum and $\mS\in \SH$ the classical sphere spectrum. Let $k$ be an algebraically closed field of characteristic zero, let $\mS_k\in \SH(k)$ denote the motivic sphere spectrum, and let $s^t_p\mS_k$ denote the $p$th layer in the slice tower for $\mS_k$. Consider the Adams-Novikov spectral sequence
\[
E_2^{s,t}(AN)=\Ext^{s, t}_{\MU_*(\MU)}(\MU_*, \MU_*)
\Longrightarrow \pi_{t-s}(\mS)
\]
and the Atiyah-Hirzebruch spectral sequence for $\pi_{*,0}\mS_k(k)$ associated to the slice tower for $\mS_k$, 
\[
E_1^{p,q}(AH)=\pi_{-p-q,0}(s^t_p\mS_k)(\Spec k)\Longrightarrow\pi_{-p-q,0}\mS_k(\Spec k). 
\]
Then there is an isomorphism  
\[
\gamma_1^{p,q}:E_1^{p,q}(AH)\cong E_2^{3p+q,2p}(AN)
\]
which induces a sequence of isomorphisms of complexes (for all $r\ge1$)
\[
\oplus_{p,q}\gamma_{r}^{p,q}: (\oplus_{p,q}E_r^{p,q}(AH), d_r)\to (\oplus_{p,q}E_{2r+1}^{3p+q,2p}(AN), d_{2r+1}).\] 
\end{IntroThm}
We remind the reader that, for an object $\sE\in \SH(k)$ and integers $p, q$,  one has the {\em homotopy sheaf} $\pi_{p,q}(\sE)$, this being the Nisnevich sheaf on smooth schemes over $k$  associated to the presheaf
 \[
 U\in \Sm/k \mapsto \Hom_{\SH(k)}(\Sigma_{S^1}^{p-2q}\Sigma_{\P^1}^q\Sigma^\infty_{\P^1} U_+, \sE).
 \]
In particular, the term $\pi_{-p-q,0}(s^t_p\mS_k)(\Spec k)$ occurring in theorem~\ref{thm:Main} is the evaluation of the presheaf 
 $\pi_{-p-q,0}(s^t_p\mS_k)$ on the final object $\Spec k$ of $\Sm/k$.
 
 We have an $\ell$-local version of theorem~\ref{thm:Main} as well:
  \begin{IntroThm} \label{thm:MainPLocal} Let $k$, $\mS$ and $\mS_k$ be as in theorem~\ref{thm:Main}.  Fix a prime $\ell$ and let $\BP^{(\ell)} \in \SH$ be the associated Brown-Peterson spectrum.  Consider the $\ell$-local Adams-Novikov spectral sequence
\[
E_2^{s,t}(AN)_\ell=\Ext^{s, t}_{\BP^{(\ell)}_*(\BP^{(p)})}(\BP^{(\ell)}_*, \BP^{(\ell)}_*)
\Longrightarrow \pi_{t-s}(\mS)\otimes\ZZ_{(\ell)}
\]
and the  $\ell$-local Atiyah-Hirzebruch spectral sequence
\[
E_1^{p,q}(AH)_\ell=\pi_{-p-q,0}(s^t_p\mS_k)(\Spec k)\otimes\ZZ_{(\ell)}\Longrightarrow\pi_{-p-q,0}\mS_k(\Spec k)\otimes\ZZ_{(\ell)}.
\]
Then there is an isomorphism  
\[
\gamma_1^{p,q}:E_1^{p,q}(AH)_\ell\cong E_2^{3p+q,2p}(AN)_\ell
\]
which induces a sequence of isomorphisms of complexes  (for all $r\ge1$)
\[
\oplus_{p,q}\gamma_{r}^{p,q}: (\oplus_{p,q}E_r^{p,q}(AH)_\ell, d_r)\to (\oplus_{p,q}E_{2r+1}^{3p+q,2p}(AN)_\ell, d_{2r+1})
\] 
\end{IntroThm}

 \begin{IntroRem} The Atiyah-Hirzebruch spectral sequence is often presented as an $E_2$-spectral sequence:
 \[
 E_2^{p,q}(AH;\sE, X)':= H^{p-q}(X, \pi^\mu_{-q}(\sE)(n-q))\Longrightarrow \sE^{p+q,n}(X).
 \]
 Here $\pi^\mu_n(\sE)$ is the {\em homotopy motive} of $\sE$, that is, a canonically determined object of $\DM(k)$ with $\EM(\pi^\mu_n(\sE)(n)[2n])\cong s_n^t\sE$. Thus, for $\sE=\mS_k$, $X=\Spec k$,  this gives
\[
E_r^{p,q}(AH)'= E_{r-1}^{-q,p+2q}(AH)
\]
and theorem~\ref{thm:Main} yields the isomorphism  (for all $r\ge2$)
\[
E_r^{p,q}(AH)'\cong E_{2r-1}^{p-q, -2q}(AN),
\]
answering affirmatively the question raised in \cite[Introduction]{LevineComp}. \end{IntroRem}

\begin{IntroRem} Using an argument based on the Adams spectral sequence, Heller and Ormsby \cite{HellerOrmsby} have proven a version of the comparison theorem of \cite{LevineComp} for real closed fields, showing  the  $C_2$-equivariant homotopy theory agrees with the motivic theory, after $p$-completion. It would be interesting to make a further comparison of associated Adams-Novikov spectral sequences along the lines of this paper.
 \end{IntroRem}

The paper is organized in a somewhat nonsequential fashion: We give the arguments for our main results in the first section, after introducing notation and extracting the necessary technical underpinnings from later sections. The referee felt, and we agree, that many of these technical results are known to the experts or are straightforward extensions of known results, or can simply be accepted as black boxes in the main arguments, subject, of course, to later verification.

We then turn to filling in these technical details.  In section~\ref{sec:Diagram}, we review two constructions of a model category structure on the functor category: the projective model structure and the Reedy model structure. We apply this material to give constructions of slice towers and Betti realizations for motivic homotopy categories associated to functor categories, and prove some useful facts relating these two constructions. In section~\ref{sec:Cosimpl} we specialize to the case of the category $\Delta$ of finite ordered sets, and recall the Bousfield-Kan functor $\Tot$ and  the associated  $\Tot$-tower and spectral sequence. In section~\ref{CosimpCube} we recall from \cite{S1, S2} how the $\Tot$-tower can be described using cubical constructions, which are technically easier to  handle. As an application, we show how applying the slice tower termwise to the truncated cosimplicial objects arising in the motivic Adams-Novikov tower gives approximations to the slice tower for the motivic sphere spectrum (proposition~\ref{prop:SliceComp}).
 
In section~\ref{sec:Decalage} we adapt Deligne's d\'ecalage construction to the setting of cosimplicial objects in a stable model category that admits a $t$-structure and associated Postnikov tower, this latter construction replacing the canonical truncation of a complex. The main comparison result is achieved in proposition~\ref{prop:Dec}. This is the technical tool that enables us to   compare the Atiyah-Hirzebruch and Adams-Novikov spectral sequences. The treatment of this topic is less than optimal, as one should expect a more general extension of Deligne's d\'ecalage construction to some version of filtered objects in a model category, but the special case handled here suffices for our main application.

We would like to thank the referees for their careful reading of an earlier version and their very helpful suggestions for improving this paper. Especially useful was the suggestion that   the paper be reorganized by putting the main arguments first. Directing my attention to the works of Sinha \cite{S1, S2}, which give the reformulation of  the truncated $\Tot$-construction as a punctured cube, allowed me replace my considerably less elegant arguments for some weaker statements in \S3 with Sinha's results.

Finally, I would like to thank the Alexander von Humboldt Foundation for their generous support.

 \begin{IntroRem}
The reader will have noticed a ``clash of spectral sequence cultures'' between the slice and Adams-Novikov spectral sequences. The slice spectral sequence follows the Cartan-Eilenberg convention, standard in homological algebra and algebraic geometry, in which the $d_r$ differential has bidegree $(r, 1-r)$, and $E_\infty^{p,q}$ contributes to $\pi_{-p-q}$ or $H^{p+q}$. The Adams-Novikov spectral sequence follows the Bousfield-Kan convention, in which the $d_r$ differential has bidegree $(r, r-1)$ and $E_\infty^{p,q}$ contributes to $\pi_{q-p}$. The evident transformation $E^{p,q}(CE)=E^{p,-q}(B.K.)$ relates the two conventions. We will use the Cartan-Eilenberg convention throughout this paper, except for the Adams-Novikov spectral sequence and its $\ell$-local version. Contrary to standard practice, we will be using the Cartan-Eilenberg convention for the spectral sequence of the Tot-tower in \S\ref{sec:Cosimpl}; we felt that using the Bousfield-Kan indexing would unnecessarily complicate the proof of proposition~\ref{prop:Dec}.  In any case, the indexing for the convergent should make the choice of convention clear. 
\end{IntroRem}

\section{The Adams-Novikov spectral sequence}\label{sec:AN}
We have buried most of the technical results needed for our argument in later sections, so as to present the main thread  in as direct a fashion as possible. As a preliminary to the main argument, we  extract the relevant aspects of these technical constructions.

We have the model category of pointed spaces (i.e., pointed simplicial sets) $\Spc_\bullet$ and the associated model category of suspension spectra $\Spt$, with homotopy categories $\sH_\bullet$ and $\SH$, respectively. We fix a perfect base field $k$, giving us the model category of pointed spaces over $k$, $\Spc_\bullet(k)$ and the model category of $T$-spectra $\Spt_T(k)$, with associated homotopy categories $\sH_\bullet(k)$ and $\SH(k)$; the model category structures are discussed in the examples~\ref{exs:FunctorialConstructions}. We have Voevodsky's slice tower in $\SH(k)$, written as a tower 
\[
\ldots\to f^t_{n+1}\to f^t_n\to \ldots\to \id
\]
of endofunctors. For a given embedding $\sigma:k\hookrightarrow\CC$, we have the associated Betti realization functor
\[
\Re_B:\SH(k)\to \SH.
\]

These constructions extend to functor categories, using  the machinery discussed in  \S\ref{sec:Diagram}. For a Reedy category $\sS$, we have the functor category $\Spt_T(k)^\sS$ with the model structure as described in \S \ref{subsec:ModelStructure}, and the associated homotopy category $\Ho\Spt_T(k)^\sS$.\footnote{Not to be confused with the functor category $(\Ho\Spt_T(k))^\sS$. To avoid additional parentheses, we will always use the notation $\Ho\sM^\sS$ for the homotopy category of a model category of functors $\sM^\sS$ and reserve $(\Ho\sM)^\sS$ for the category of functors to the homotopy category.}  The slice endofunctors $f^t_n$ extend to the functor category to give a tower of endofunctors
\[
\ldots\to f^{\sS,t}_{n+1}\to f^{\sS,t}_n\to \ldots\to \id
\]
of $\Ho\Spt_T(k)^\sS$. The two slice towers are compatible for each $s\in \sS$  via the evaluation functor $i_{s*}:\Ho\Spt_T(k)^\sS\to \SH(k)$. Similarly, we have the functor model category $\Spt^\sS$ and  the associated Betti realization functor
\[
\Re^\sS_B:\Ho\Spt_T(k)^\sS\to \Ho\Spt^\sS,
\]
again, compatible with $\Re_B$ via the evaluation functors $i_{s*}$ for each $s\in \sS$.

We have a parallel situation with respect to the classical Postnikov tower in $\SH$:
\[
\ldots\to f_{n+1}\to f_n\to \ldots\to \id,
\]
where for a spectrum $E$, $f_nE\to E$ is the usual $(n-1)$-connected cover. This passes to the functor category, giving the tower of endofunctors of $\Ho\Spt^\sS$
\[
\ldots\to f^\sS_{n+1}\to f^\sS_n\to \ldots\to \id,
\]
compatible with the classical Postnikov tower via the evaluation functors.

We denote the homotopy cofiber of $f^{\sS,t}_b\to f^{\sS,t}_a$ by $f^{\sS,t}_{a/b}$ and define $f^{\sS}_{a/b}$ similarly.

The second type of tower that we will need is the $\Tot$-tower associated to a cosimplicial object in a suitable model category (see \S\ref{sec:Cosimpl} for details). 
For a pointed model category $\sM$ we give the functor category $\sM^\Delta$ the Reedy model structure. We consider the functor $\Tot:\sM^\Delta\to \sM$ which sends a cosimplicial object $\sX:\Delta\to \sM$ to $\Tot(X):=\sHom(\Delta[*], \sX)$, where $\Delta[*]$ is the cosimplicial space $n\mapsto \Delta[n]$. Restricting $\Delta[n]$ to its $k$-skeleton gives the cosimplicial space  $n\mapsto \Delta[n]^{(k)}$ and the associated object $\Tot_{(k)}(\sX):=\sHom(\Delta[*]^{(k)}, \sX)$, which fit together to give the Tot-tower
\[
\Tot(\sX)\to \ldots\to \Tot_{(k+1)}(\sX)\to \Tot_{(k)}(\sX)\to \ldots\to \Tot_{(0)}(\sX)\to \Tot_{(-1)}(\sX)=\pt.
\]

The slice tower   and $\Tot$-tower generate spectral sequences; see \S\ref{subsec:TotSSeq} for a discussion of the $\Tot$-tower spectral sequence. The slice tower spectral sequence gives rise to a {\em motivic Atiyah-Hirzebruch spectral sequence} $E(AH)$, and the $\Tot$-tower spectral sequence $E(\Tot, \sX)$ applied to the cosimplicial spectrum $n\mapsto \MU^{\wedge n+1}$ gives the  Adams-Novikov spectral sequence $E(AN)$. The classical Postnikov tower yields the classical Atiyah-Hirzebruch spectral sequence. 

We may combine the $\Tot$ construction with the Postnikov tower as follows: Given a cosimplicial spectrum $m\mapsto \sX^m\in \Spt$, we can apply the functor $f_n$ termwise, giving the tower 
\[
\ldots\to f^\Delta_{n+1}\sX\to f^\Delta_n\sX\to \ldots\to \sX
\]
in $\Spt^\Delta$; we may then apply the functor $\Tot$, giving us the tower
\[
\ldots \to \Tot f^\Delta_{n+1}\sX\to \Tot f^\Delta_n \sX\to \ldots \to \Tot \sX
\]
 in $\Spt$. We call the resulting spectral sequence the {\em d\'ecalage} of the $\Tot$-spectral sequence, and denote this by $E(\Dec, \sX)$.  Our main technical result concerning these spectral sequences is Proposition~\ref{prop:Dec}, which states that the spectral sequences $E(\Dec, \sX)$ and $E(\Tot, \sX)$ are equal up to a reindexing.

We now apply these tools to the problem of comparing the slice spectral sequence for the motivic sphere spectrum with the Adams-Novikov spectral sequence for the topological sphere spectrum.

We will call a spectral sequence  $(E^{p,q}_r, d_r)$   {\em strongly convergent to $\Pi_*$} if it is bounded and converges to $\Pi_*$ (see \cite[\S 5.2.5]{WeibelHA}). Under the Cartan-Eilenberg indexing convention, this means simply that  for each $n$ there is a finite, exhaustive and separated filtration $F_*\Pi_n$ of $\Pi_n$, and an integer $r(n)$ such that $\gr_F^p\Pi_n\cong E^{p, n-p}_\infty=E^{p, n-p}_r$ for all $p$ and all $r\ge r(n)$. We call a spectral sequence $E_a^{p,q}\Longrightarrow \Pi_{-p-q}$   strongly convergent if the spectral sequence is strongly convergent to $\Pi_*$.

Let  $k$ an algebraically closed field with an embedding $\sigma:k\hookrightarrow\CC$ and let $\mS_k\in \Spt_T(k)$ be the motivic sphere spectrum. Let $\mS\in \Spt$ denote the classical sphere spectrum. The Betti realization of the slice tower for $\mS_k$ gives a tower in $\SH$
\[
\ldots\to \Re_B f^t_{n+1}\mS_k\to \Re_B f^t_n\mS_k\to  \ldots\to \Re_B f^t_0\mS_k=\mS,
\]
with $n$th layer equal to $\Re_B s^t_n\mS_k$. This gives the spectral sequence
\[
E_2^{p,q}(AH)'=\pi_{-p-q}\Re_B s^t_{-q}\mS_k\Longrightarrow \pi_{-p-q}\mS.
\]
One has a similar spectral sequence using the slice tower itself
\[
\ldots\to f^t_{n+1}\mS_k\to   f^t_n\mS_k\to  \ldots\to  f^t_0\mS_k,
\]
namely,
\[
E_2^{p,q}(AH)'_{\mot}=\pi_{-p-q,0}(s^t_{-q}\mS_k)(k)\cong H^{p-q}(k,\pi^\mu_{-q}(-q))\Longrightarrow \pi_{-p-q,0}(\mS_k)(k).
\]
We have shown in \cite[theorem 4]{LevineConv} that this latter spectral sequence is strongly convergent and in \cite[proposition 6.4,  theorem 6.7]{LevineComp} that the Betti realization functor gives an isomorphism of $ \pi_{-p-q,0}(\mS_k)(k)$ with $\pi_{-p-q}\mS$, as well as an isomorphism of the spectral sequence $E(AH)'_{\mot}$ with the spectral sequence $E(AH)'$. One can thus say that the spectral sequence $E(AH)'$ is of ``motivic origin''.

In addition, in \cite[theorem 4]{LevineComp} we identified $E_2^{p,q}(AH)'$ with an $E_2$ term of the Adams-Novikov spectral sequence for $\mS$:
\[
E_2^{p,q}(AH)'\cong E_2^{p-q,2q}(AN).
\]
We wish to extend this result by showing that the spectral sequence $E(AH)'$ agrees with the Adams-Novikov spectral sequence $E(AN)$, after a suitable reindexing. 

In principle, the argument should go like this:  Let  $\tilde{\MU}$ be a strict monoid object in symmetric spectra representing the usual $\MU$ in $\SH$. Let $\tilde{\MU}^{\wedge *+1}$ be the cosimplicial (symmetric) spectrum $n\mapsto \tilde{\MU}^{\wedge n+1}$ with the $i$th coface map inserting the unit map in $i$th spot, and the  $i$th codegeneracy map taking the product of the $i$th and $(i+1)$st factors. The Adams-Novikov spectral sequence is just the $\Tot$-tower spectral sequence \eqref{eqn:TowerSS} associated to the cosimplicial symmetric spectrum $\tilde{\MU}^{\wedge *+1}$. Let $\tilde\MGL^{\wedge *+1}$ be the motivic analog to $\tilde{\MU}^{\wedge *+1}$, giving us a cosimplicial $T$-spectrum $n\mapsto \tilde\MGL^{\wedge n+1}$, with coface and codegeneracy maps defined as for $\tilde{\MU}^{\wedge *+1}$. One could hope to have a ``total $T$-spectrum functor" $\Tot:\Spt_T(k)^\Delta\to \Spt_T(k)$ and weak equivalences
\[
\mS_k\cong \Tot\tilde\MGL^{\wedge *+1},\quad f_p^t\mS_k\cong \Tot f_p^t\tilde\MGL^{\wedge *+1},
\]
where $f_p^t\tilde\MGL^{\wedge *+1}$ is the cosimplicial spectrum $n\mapsto f_p^t\tilde\MGL^{\wedge n+1}$, using a suitable functorial model for $f^t_p$ in $\Spt^\Sigma_T(k)$. 

The layers of $\tilde\MGL^{\wedge n+1}$ for the slice filtration are known by work of Hopkins-Morel (see \cite{HopkinsMorel}) and Hoyois \cite{Hoyois}, and one can show that the Betti realization of the slice  $s_p^t\tilde\MGL^{\wedge n+1}$ is just the $2p$th layer $f_{2p/2p+1}\MU^{\wedge n+1}$ in the Postnikov tower for $\MU^{\wedge n+1}$. Thus, one could hope  to have an isomorphism in $\Ho\Spt^\Delta$
\[
\Re_B f_p^{\Delta,t}\tilde\MGL^{\wedge *+1}\cong f^\Delta_{2p-1}\tilde{\MU}^{\wedge *+1}\cong f^\Delta_{2p}\tilde{\MU}^{\wedge *+1}.
\]

After changing the $E_2$ Atiyah-Hirzebruch spectral sequence to an $E_1$ spectral sequence
\[
E_1^{p,q}(AH):=\pi_{-p-q,0}(s_p^t\mS_k)(k)\Longrightarrow \pi_{-p-q,0}\mS,
\]
we would then have an isomorphism 
\[
E_1^{p,q}(AH)\cong E_1^{2p,q-p}(\Dec, \tilde{\MU}^{\wedge *}) 
\]
leading to the isomorphisms
\[
E_r^{p,q}(AH)\cong E_{2r-1}^{2p,q-p}(\Dec, \tilde{\MU}^{\wedge *+1})\cong E_{2r}^{2p,q-p}(\Dec, \tilde{\MU}^{\wedge *})
\]
and corresponding isomorphisms of complexes.

Using proposition~\ref{prop:Dec} (for spectra) would then give the sequence of isomorphisms
\[
E_r^{p,q}(AH)\cong E_{2r+1}^{3p+q,2p}(AN)
\]
and corresponding isomorphisms of complexes. This would then give the isomorphisms 
\[
E_r^{p,q}(AH)'\cong E_{2r-1}^{p-q, -2q}(AN).
\]
for all $r\ge2$.

We prefer to avoid the technical problems arising from the compatibility of the Betti realization with the functor $\Tot$,  and from checking if $\mS_k\to \Tot\tilde\MGL^{\wedge *+1}$ is an isomorphism; instead we work with the approximations $\Tot_{(N)}\tilde\MGL^{\wedge *+1}$ and 
 $\Tot_{(N)}\tilde{\MU}^{\wedge *+1}$. These suffice to give the desired isomorphisms of $E_r$-complexes, by simply taking $N$ sufficiently large and using  proposition~\ref{prop:SliceComp} to show that the truncation $\Tot_{(N)}\tilde\MGL^{\wedge *+1}$ approximates $\mS_k$ sufficiently well with respect to the slice tower.

 We deal with the compatibility of $\Tot_{(N)}$ and $\Re_B$ as follows: for a cosimplicial object $n\mapsto \sX^n$ in a pointed simplicial model category $\sM$, we have an associated punctured $(N+1)$-cube 
  \[
 \phi^{N+1}_{0*}\sX:\sq_0^{N+1}\to \sM.
 \]
 See \eqref{eqn:CubeFunctor} for the definition; this arises by identifying the punctured $(N+1)$-cube $\sq_0^{N+1}$ with the category $\Delta_{inj}/[N]$, where $\Delta_{inj}\subset \Delta$ is the subcategory of injective maps. In addition, we have an isomorphism in $\Ho\sM$ (proposition~\ref{prop:SimpVCube})
 \[
 \holim_{\sq_0^{N+1}} \phi^{N+1}_{0*}\sX\cong \Tot_{(N)}\sX.
 \]
Moreover, using the isomorphism \eqref{eqn:SqInduction}, one can describe $\holim_{\sq_0^{N+1}} \phi^{N+1}_{0*}\sX$ as an iterated  homotopy pullback.  As the functor $\Re_B$ is exact, it is compatible (up to weak equivalence) with taking an iterated homotopy pullback, giving us a canonical isomorphism 
 \begin{equation}\label{eqn:ReComp}
 \Re_B(\Tot_{(N)}(\sE))\cong \Tot_{(N)}(\Re_B^\Delta(\sE))
 \end{equation}
  in $\SH$ for $\sE:\Delta\to \Spt_T(k)$ a cosimplicial $T$-spectrum.  
 
 We drop the tilde from the notation, considering both $\MU$ and $\MGL$ as objects in the appropriate category of symmetric spectra. We will also systematically replace objects with their fibrant replacements without altering the notation, so all objects that appear below will be assumed to be fibrant (or even cofibrant and fibrant, if need be).

We have the cosimplicial objects 
\[
\MGL^{\wedge *+1}\in \Spt_T^\Sigma(k)^\Delta,\
\MU^{\wedge *+1}\in (\Spt^\Sigma)^\Delta.
\]
Restricting to the subcategory $\Delta^{\le N}$ via the inclusion functor $\iota_N:\Delta^{\le N}\to \Delta$
gives us the  truncated objects
\[
\iota_{N*}\MGL^{\wedge *+1}\in \Spt_T^\Sigma(k)^{\Delta^{\le N}},\ 
\iota_{N*}\MU^{\wedge *+1}\in (\Spt^\Sigma)^{\Delta^{\le N}}. 
\]
As the Betti  realization of $\MGL$ is isomorphic to $\MU$ and $\Re_B$ is an exact monoidal functor, we have the isomorphism in $\Ho[(\Spt^\Sigma)^{\Delta^{\le N}}]$
\[
\Re_B^{\Delta^{\le N}}\iota_{N*}\MGL^{\wedge *+1}\cong \iota_{N*} \MU^{\wedge *+1}.
\]
Our main task is to identify the tower 
\begin{multline*}
\ldots\to \Re_B^{\Delta^{\le N}}f_{n+1}^{\Delta^{\le N},t}\iota_{N*} \MGL^{\wedge *+1}\to\Re_B^{\Delta^{\le N}}f_n^{\Delta^{\le N},t}\iota_{N*} \MGL^{\wedge *+1}\\
\to\ldots\to
\Re_B^{\Delta^{\le N}}\iota_{N*} \MGL^{\wedge *+1}.
\end{multline*}

As notation, for $\sE\in \Spt_T(k)$, $I=(i_1,\ldots, i_r)$ an index with $0\le i_j\in \ZZ$,  $b_I=b_1^{i_1}\cdot\ldots\cdot b_r^{i_r}$ a monomial, with $b_j$ of degree $n_j$, we define 
\[
\sE\cdot b_I:=\Sigma_T^{|I|}\sE, 
\]
where $|I|:=\sum_{j=1}^r n_j\cdot i_j$.  More generally, if $\{b^i_j\}$ is a set of variables, $i=1,\ldots m$, with some assigned positive integral degrees,  we let $\sE[\{b_j^i\}]$ denote the coproduct of the $\sE b^1_{I_1}\cdot \ldots\cdot b^m_{I_m}$.

\begin{lem}\label{lem:MGLDecomp} We have an isomorphism of left $\MGL$-modules
\[
\MGL^{\wedge m+1}\cong  \MGL[b_\bullet^{(1)},\ldots, b_\bullet^{(m)}]
\]
where $b_\bullet^{(j)}$ is the collection of variables $b_1^{(j)}, b_2^{(j)},\ldots$, with $b_n^{(j)}$ of degree $n$. 
\end{lem}

\begin{proof} We write $\pi_{*,*}\MGL$ for  $\pi_{*,*}\MGL(k)$, etc.  It clearly suffices to handle the case $m=1$. For this, \cite[proposition 6.2]{NSO} gives us elements $b_n\in \pi_{2n,n}(\MGL\wedge\MGL)$ giving rise to an isomorphism of left $\pi_{*,*}\MGL$-modules
\[
\pi_{*,*}(\MGL\wedge\MGL)\cong\pi_{*,*}\MGL[b_1, b_2,\ldots].
\]
For each monomial $b_I$ in $b_1, b_2, \ldots$, we view $b_I\in \pi_{2|I|,|I|}(\MGL\wedge\MGL)$ as a  map $b_I:\Sigma^{|I|}_T\mS_k\to \MGL\wedge \MGL$; using the product in $\MGL$, this gives us the left $\MGL$-map
\[
\vartheta:=\sum_I b_I:\oplus_I\Sigma^{|I|}_T\MGL\to \MGL\wedge \MGL.
\]
Now, $\MGL$ is stably cellular \cite[theorem 6.4]{DuggerIsaksen} hence $\oplus_I\Sigma^{|I|}_T\MGL$ and  $\MGL\wedge\MGL$ are stably cellular (the second assertion follows from \cite[lemma 3.4]{DuggerIsaksen}). Clearly $\vartheta$ induces an isomorphism on $\pi_{a,b}$ for all $a,b$, hence by \cite[corollary 7.2]{DuggerIsaksen}, $\vartheta$ is an isomorphism in $\SH(k)$. 
\end{proof}

\begin{lem}\label{lem:MGLTower}  (1)  $\Re_B(f_n^t\MGL^{\wedge m+1})$ is $(2n-1)$-connected for all $n, m\ge0$.\\
(2) The map 
\[
f_{2n}\Re_B(f_n^t\MGL^{\wedge m+1})\to f_{2n}\Re_B(\MGL^{\wedge m+1})
\]
induced by the natural transformation $f_n^t\to \id$
 and the map  
 \[
 f_{2n}\Re_B(f_n^t\MGL^{\wedge m+1})\to \Re_B(f_n^t\MGL^{\wedge m+1})
 \]
induced by the natural transformation $f_{2n}\to\id$ are weak equivalences.\\
(3) The  map 
\[
f_{2n}^{\Delta^{\le N}}\Re^{\Delta^{\le N}}_B(f_n^{\Delta^{\le N},t}\iota_{N*}\MGL^{\wedge *+1})\to f_{2n}^{\Delta^{\le N}}\Re_B\iota_{N*}\MGL^{\wedge *+1}
\]
induced by the natural transformation $f_n^{\Delta^{\le N},t}\to\id$ and the map
\[
f_{2n}^{\Delta^{\le N}}\Re^{\Delta^{\le N}}_B(f_n^{\Delta^{\le N},t}\iota_{N*}\MGL^{\wedge *+1})\to \Re^{\Delta^{\le N}}_B(f_n^{\Delta^{\le N},t}\iota_{N*}\MGL^{\wedge *+1})
\]
induced by the natural transformation $f_{2n}^{\Delta^{\le N}}\to\id$ are weak equivalences.
\end{lem}

\begin{proof} We first prove (1). Recall that a $\P^1$-spectrum $\sE$ is  {\em topologically $c$-connected} if the homotopy sheaf $\pi_{n+m,m}\sE$ is zero for all $n\le c$ and all $m\in \ZZ$.  It follows from Morel's $\A^1$-connectedness theorem \cite{MorelConn} that $\pi_{a+b,b}\MGL_n=0$  for $a<n$, $b\ge 0$. Thus the stable homotopy sheaves $\pi_{a+b,b}\MGL$ are zero for $a<0$, that is, $\MGL$ is topologically -1-connected. By  \cite[proposition 3.2]{LevineGW}   $f_n^t\MGL$ is also topologically -1-connected, hence by \cite[theorem 5.2]{LevineComp}  $\Re_B(f_n^t\MGL)$ is $(n-1)$-connected for all $n\ge0$.

We have an isomorphism (of left $\MGL$-modules)
\begin{equation}\label{eqn:MGLDecomp2}
\MGL^{\wedge m+1}\cong \oplus_{I=(i_1,\ldots, i_m)}\Sigma_T^{|I|}\MGL
\end{equation}
from which it follows that $f_n^t\MGL^{\wedge m+1}$ is topologically -1-connected and that $\Re_Bf_n^t\MGL^{\wedge m+1}$  is $(n-1)$-connected for all $n\ge0$. Thus the spectral sequence associated to the tower
\[
\ldots\to \Re_Bf_{N+1}^t\MGL^{\wedge m+1}\to \Re_Bf_N^t\MGL^{\wedge m+1}\to\ldots\to \Re_Bf_n^t\MGL^{\wedge m+1} 
\]
converges strongly to the homotopy groups of $\Re_Bf_n^t\MGL^{\wedge m+1}$. As both $f_N^t$ and $\Re_B$ are exact functors, the $\ell$th layer in this tower is $\Re_Bs_{n+\ell}^t\MGL^{\wedge m+1}$, so to prove (1), it suffices to show that $\Re_Bs_{n+\ell}^t\MGL^{\wedge m+1}$ is $(2n-1)$-connected for all $\ell\ge0$.

By the Hopkins-Morel-Hoyois theorem \cite{Hoyois, HopkinsMorel} and the above computation of  $\MGL^{\wedge m+1}$, $s^t_N\MGL^{\wedge m+1}$ is a finite coproduct of copies of $\Sigma_T^NM\ZZ$, where $M\ZZ$ is the motivic Eilenberg-MacLane spectrum representing motivic cohomology. In addition, $\Re_B(M\ZZ)\cong EM(\ZZ)$, hence $\Re_Bs^t_N\MGL^{\wedge m+1}$ is a finite coproduct of copies of $\Sigma^{2N}EM(\ZZ)$, and is thus $(2N-1)$-connected. 

For (2), applying $\Re_B$ to the decomposition \eqref{eqn:MGLDecomp2} gives
\[
\Re_B(\MGL^{\wedge m+1})\cong \oplus_I\Sigma^{2|I|}\MU;
\]
since $f_n^t\circ\Sigma_T\cong \Sigma_T\circ f_{n-1}^t$, and $f_{m}\circ \Sigma\cong \Sigma\circ f_{m-1}$, this reduces the proof of (2) to the case $m=0$. Since 
\[
\Re_Bs_N^t\MGL\cong \Sigma^{2N}EM(\ZZ)\otimes \MU^{-2N}
\]
the strongly convergent spectral sequences
\begin{align*}
&E_1^{p,q}=\pi_{-p-q}\Re_Bs_p^t\MGL\Longrightarrow \pi_{-p-q}\Re_B\MGL,\\
&E_1^{p,q}=\pi_{-p-q}\Re_Bs_p^t\MGL\Longrightarrow \pi_{-p-q}\Re_Bf_n^t\MGL
\end{align*}
degenerate at $E_1$ and show that $\pi_m\Re_Bf_n^t\MGL\to \pi_m\Re_B\MGL$ is an isomorphism for $m\ge 2n$ and $\pi_m\Re_Bf_n^t\MGL=0$ for $m<2n$. Thus $\Re_Bf_n^t\MGL\to \Re_B\MGL\cong \MU$ is isomorphic (in $\SH$) to the $(2n-1)$-connected cover of $\MU$, proving (2).

(3) follows immediately from (2),  by the definition of the weak equivalences in the functor category $\sM^\sS$. \end{proof}

Having gone through these preliminaries, we can now prove our main result:
 
 \begin{proof}[Proof of theorem~\ref{thm:Main}] We replace all objects with fibrant models, without changing the notation. 

 Denote the spectral sequence  \eqref{eqn:TowerSS2} for indices $A=0$, $B= \infty$ and cosimplicial spectrum $\sE$  by $E(\Tot, \sE)$. We let $E(\Dec, \sE)$ denote the spectral sequence \eqref{eqn:DecTowerSS} for  $A=0$, $B= \infty$. The Adams-Novikov spectral sequence may be constructed as the spectral sequence $E(\Tot, \MU^{\wedge *+1})$ associated to the cosimplicial spectrum
\[
n\mapsto \MU^{\wedge n+1}.
\]

For $k\subset K$ an extension of algebraically closed fields, the base extension induces an isomorphism of the spectral sequence $E(AH)$ for $k$ with the spectral sequence $E(AH)$ for $K$; this follows  from e.g. \cite[theorem 8.3]{LevineComp}. Thus, we may assume that $k$ admits an embedding into $\CC$, giving the associated Betti realization functor
\[
\Re_B:\SH(k)\to \SH.
\]
 
By lemma~\ref{lem:MGLTower} and the isomorphism \eqref{eqn:ReComp},  we have an isomorphism in $\SH$
\[
\Re_B(\Tot_{(N)}f^{\Delta, t}_{a/b}\MGL^{\wedge *+1})\cong \Tot_{(N)}f^\Delta_{2a/2b}\MU^{\wedge *+1}
\]
for all $a\le b$, including $b=\infty$, compatible with respect to the maps in the slice tower for  $\MGL^{\wedge *+1}$ and the Postnikov tower for
$\MU^{\wedge *+1}$.

By proposition~\ref{prop:SliceComp}, this  gives us an isomorphism in 
$\SH$
\begin{equation}\label{eqn:WEq}
\Re_B(f^t_{a/b}\mS_k)\cong \Tot_{(N)}f^\Delta_{2a/2b}\MU^{\wedge *+1}
\end{equation}
for $a\le b\le N+1$, compatible with respect to change in $a$ and $b$. In addition, as we have replaced $f_{2a/2b}^\Delta\MU^{\wedge *+1}$ with a fibrant model, the Tot-tower for $f^\Delta_{2a/2b}\MU^{\wedge *+1}$ is a tower of fibrations and
\[
\Tot f^\Delta_{2a/2b}\MU^{\wedge *+1}\cong \lim_N\Tot_{(N)}f^\Delta_{2a/2b}\MU^{\wedge *+1}
\cong \holim_N\Tot_{(N)}f^\Delta_{2a/2b}\MU^{\wedge *+1}
\]

Using the isomorphism \eqref{eqn:WEq} gives us the isomorphism in 
$\SH$
\[
\Re_B(f^t_{a/b}\mS_k)\cong \Tot f^\Delta_{2a/2b}\MU^{\wedge *+1},
\]
and since  the functor $\Tot$ is compatible with homotopy fiber sequences, we have isomorphisms in $\SH$
\[
\Tot f^\Delta_{2a/2b}\MU^{\wedge *+1}\cong \hocofib\left( \Tot f^\Delta_{2b}\MU^{\wedge *+1}\to \Tot f^\Delta_{2a}\MU^{\wedge *+1}\right)
\]
for all $a\le b$.  Thus, we have an isomorphism of the spectral sequence associated to the tower 
\[
\ldots\to \Re_B(f^t_{n+1}\mS_k)\to \Re_B(f^t_{n}\mS_k)\to \ldots\to \Re_B(f^t_{0}\mS_k)\cong \mS
\]
and the one associated to the tower 
\begin{multline*}
\ldots\to \Tot f^\Delta_{2n+2}\MU^{\wedge *+1}\to \Tot f^\Delta_{2n}\MU^{\wedge *+1}\to\ldots\\
\to \Tot f^\Delta_{0}\MU^{\wedge *+1}=\Tot\,\MU^{\wedge *+1}.
\end{multline*}
Since all the odd homotopy groups of $\MU^{\wedge m+1}$  vanish,  this latter  spectral sequence is just $E(\Dec, \MU^{\wedge *+1})$,  after reindexing.

By \cite[proposition 6.4]{LevineComp}, the functor $\Re_B$ induces an isomorphism
\[
 \pi_{n,0}(s^t_m\mS_k)(\Spec k)
\cong \pi_n(\Re_B(s^t_m\mS_k))
\]
for all $n$ and $m$. In addition,  the tower
\[
\ldots\to f_{m+1}^t\mS_k\to f_{m}^t\mS_k\to \ldots\to f_{0}^t\mS_k=\mS_k
\]
and its Betti realization
\[
\ldots\to \Re_Bf_{m+1}^t\mS_k\to \Re_B f_{m}^t\mS_k\to \ldots\to \Re_Bf_{0}^t\mS_k=\mS
\]
yield strongly convergent spectral sequences
\begin{align*}
&E_1^{p,q}=\pi_{-p-q,0}(s^t_p\mS_k)(\Spec k)\Longrightarrow \pi_{-p-q,0}(f^t_{a/b}\mS_k)(\Spec k)
\\
&E_1^{p,q}=\pi_{-p-q}\Re_B(s^t_p\mS_k)\Longrightarrow \pi_{-p-q}\Re_B(f^t_{a/b}\mS_k)
\end{align*}
(see \cite[theorem 4]{LevineConv}, \cite[proof of theorem 6.7]{LevineComp})
and thus   the functor $\Re_B$ induces an isomorphism
\[
 \pi_{n,0}(f^t_{a/b}\mS_k)(\Spec k)
\cong \pi_n(\Re_B(f^t_{a/b}\mS_k))
\]
for all $n$ and all $a<b\le \infty$. 

Putting these two pieces together, the Betti realization functor gives an isomorphism  of the spectral sequence $E(AH)$ with the spectral sequence $E(\Dec, \MU^{\wedge *+1})$, after a suitable reindexing. Explicitly, this gives 
\[
E_1^{p, q}(AH)\cong E_1^{2p,q-p}(\Dec, \MU^{\wedge *+1})=E_2^{2p,q-p}(\Dec, \MU^{\wedge *+1});
\]
the terms $E_*^{p,q}(\Dec, \MU^{\wedge *+1})$ with $p$ odd are all zero, and by induction, we have  isomorphisms
\[
E_r^{p, q}(AH)\cong E_{2r}^{2p,q-p}(\Dec, \MU^{\wedge *+1})
\]
commuting with  the differentials $d_r(AH)$ and $d_{2r}(\Dec)$.  We apply proposition~\ref{prop:Dec} to yield the isomorphism
\[
\gamma_r^{p,q}:E_r^{p,q}(\Dec,  \MU^{\wedge *+1})\to E_{r+1}^{2p+q,-p}( \MU^{\wedge *+1})=E_{r+1}^{2p+q, p}(AN).
\]
(The change in indices in the last identity results from passing from the Cartan-Eilenberg indexing convention to that of Bousfield-Kan). This completes the proof.
 \end{proof}
 
 \begin{rems} 1. Fixing a prime $\ell$, one can replace $\MU$ with the $\ell$-local Brown-Peterson spectrum $\BP$, and similarly replace $\MGL$ with the motivic counterpart $\BP_{\mot}$ to $\BP$ (see e.g. \cite{Vezzosi} for a construction). Having made this substitution, repeating the above argument gives a comparison of the $\ell$-local Adams-Novikov spectral sequence with the $\ell$-localized slice spectral sequence. Indeed, the only point one needs to check is that the unit map $\mS_k\otimes \ZZ_{(\ell)}\to \BP_{\mot}$ has cofiber in $\Sigma_T\SH^\eff(k)$. For this, we note that $\BP_{\mot}$  is a summand of $\MGL\otimes \ZZ_{(\ell)}$, hence $\BP_{\mot}$ is in $\SH^\eff(k)$. Thus, we need only see that the unit map induces an isomorphism on the 0th slice, $s^t_0\mS_k\otimes \ZZ_{(\ell)}\to s^t_0\BP_{\mot}$. By Spitzweck's computation of the slices of Landweber exact theories \cite{Spitzweck}, we see that the canonical map $\MGL\to \BP_{\mot}$ induces an isomorphism $s^t_0\MGL\otimes\ZZ_{(\ell)}\to s^t_0\BP_{\mot}$; as we have already seen that $s^t_0\mS_k \to s^t_0\MGL$ is an isomorphism, the point is checked. This yields theorem~\ref{thm:MainPLocal}.\\\\
 2. Presumably the main results presented here have an analog in \'etale homotopy theory with respect to a suitable \'etale realization functor, see for example \cite{Isaksen}. Using the \'etale theory should enable an extension of these results to fields of positive characteristic, at least for the prime-to-characteristic parts of the groups concerned, but we have not checked this.
 \end{rems}
 
 \section{Constructions in functor categories} \label{sec:Diagram} It is convenient to perform constructions, such as Postnikov towers in various settings, or realization functors, in functor categories. This can be accomplished in a number of ways.  The Postnikov towers may be constructed via cofibrant replacements associated to a right Bousfield localization; by making the cofibrant replacement functorial, this extends immediately to functor categories. The Betti realization is similarly accomplished as the left derived functor of a left Quillen functor, so again, applying this functor to a functorial cofibrant replacement extends the Betti realization to a realization functor between functor categories. However, it is often useful to have more control over these constructions, which can be achieved through a full extension to the appropriate model category structure on the functor category; we give some details of this approach here. None of this basic material is new; it is assembled from \cite{Barwick, Hirschhorn, Hovey} and collected here for the reader's convenience. The applications to the slice tower and Betti realization are new.

\subsection{Model structures on functor categories}\label{subsec:ModelStructure}
Let $\sS, \sT$ be  small categories, $\sM$ a complete and cocomplete category, $\sM^\sS$ the category of functors $\sX:\sS\to\sM$. For $f:\sT\to \sS$ a functor, we have the restriction functor $f_*:\sM^\sS\to \sM^\sT$, $f_*\sX:=\sX\circ f$, with left adjoint $f^*$ and right adjoint $f^!$. For $\sX\in \sM^\sT$, we have that $f^*\sX$ and $f^!\sX$ are respectively  the left and right Kan extension in the following diagram:
\[\xymatrix{
\sT\ar[r]^\sX\ar[d]_f&\sM\\\sS}
\]
In particular, for   $s\in \sS$,  we have $i_s:\pt\to \sS$,  the inclusion functor with value $s$, inducing the evaluation functor $i_{s*}:\sM^\sS\to \sM$, the left adjoint $i_s^*:\sM\to\sM^\sS$, and the right adjoint  $i_s^!:\sM\to \sM^\sS$.

We take  $\sM$ to be a simplicial model category and consider two model structures on $\sM^\sS$.  If $\sM$ is cofibrantly generated, we may give $\sM^\sS$ the projective model structure, that is, weak equivalences and fibrations are defined pointwise, and cofibrations are characterized by having the left lifting property with respect to trivial fibrations.

In case $\sS$ is a Reedy category, one can also give $\sM^\sS$ the Reedy model structure. We first recall the definition of a Reedy category $\sS$: There is an ordinal $\lambda$,  a function (called {\em degree}) $d:\text{Obj}\sS\to \lambda$ and two subcategories $\sS_+$, $\sS_-$, such that all nonidentity morphisms in $\sS_+$ increase the degree, all nonidentity morphisms in $\sS_-$ decrease the degree, and each morphism $f$ in $\sS$ admits a unique factorization $f=a\circ b$ with $a\in \sS_+$, $b\in \sS_-$.  For $s\in \sS$, we let $\sS_-^s$ be the category of nonidentity morphisms $s\to t$ in $\sS_-$, and let $\sS_+^s$ be the category of nonidentity morphisms $t\to s$ in $\sS_+$.  Given an object $\sX\in \sM^\sS$, and $s\in \sS$, we have the {\em latching space} $L^s\sX$ and {\em matching space} $M^s\sX$: 
\[
L^s\sX:=\colim_{t\to s\in \sS_+^s}\sX(t),\ M^s\sX:=\lim_{s\to t\in \sS_-^s}\sX(t), 
\]
with the canonical morphisms $L^s\sX\to \sX(s)$, $\sX(s)\to M^s\sX$. \footnote{Note the shift in indices from the notation in \cite[Chap. X, \S 4]{BK}.}

The Reedy model structure on $\sM^\sS$ has weak equivalences the maps $f:\sX\to \sY$ such that $f(s):\sX(s)\to \sY(s)$ is a weak equivalence in $\sM$ for all $s\in \sS$, fibrations the maps $f:\sX\to \sY$ such that $\sX(s)\to \sY(s)\times_{M^s\sY}M^s\sX$ is a fibration in $\sM$ for all $s\in \sS$ and cofibrations the maps $f:\sX\to \sY$ such that $\sX(s)\amalg_{L^s\sX}L^s\sY\to \sY(s)$ is a cofibration for all $s\in \sS$. This makes $\sM^\sS$ a model category without any additional conditions on $\sM$. 

 In both of these two model structures, the evaluation functor $i_{s*}$ preserves fibrations, cofibrations and weak equivalences, and admits $i_s^*$ as left Quillen functor and $i_s^!$ as right Quillen functor. 

\begin{rem}\label{rem:CellularProper}
Suppose $\sM$ is cofibrantly generated.  If $\sS$ is a direct category, these two model structures agree; if $\sS$ is a general Reedy category,  the weak equivalences in the two model structures agree,  every fibration for the Reedy model structure is a fibration in the projective model structure, and thus every cofibration in the projective model structure is a cofibration in the Reedy model structure. Furthermore, the projective model structure is also cofibrantly generated, and is cellular, resp. combinatorial, if $\sM$ is cellular, resp. combinatorial; we refer the reader to \cite[theorem 11.6.1, theorem 12.1.5]{Hirschhorn}, \cite[theorem 2.14]{Barwick} for proofs of these assertions.  The Reedy model structure likewise inherits the combinatorial property from $\sM$ \cite[lemma 3.33]{Barwick}.

Left and right properness are similarly passed on from $\sM$ to the projective model structure on $\sM^\sS$ \cite[proposition 2.18]{Barwick}. For the Reedy model structure, the inheritance of left and right properness is proven in \cite[lemma 3.24]{Barwick}.
\end{rem}

\begin{ex} The classical example of a Reedy category is the category of finite ordered sets.
Let $\Delta$ denote the category with objects the finite ordered sets $[n]:=\{0,\ldots, n\}$ with the standard order, $n=0, 1, \ldots$, and with morphisms the order preserving maps of sets.  For a category $\sC$ the functor category $\sC^\Delta$ (resp. $\sC^{\Delta^\op}$) is called as usual the category of cosimplicial (resp. simplicial) objects in $\sC$.

We let $\Delta_{inj}$, $\Delta_{surj}$ denote the subcategories of $\Delta$ with the same objects, and with morphisms the injective, resp. surjective order preserving maps. Taking $\Delta_+:=\Delta_{inj}$, $\Delta_-:=\Delta_{surj}$ and $d:\Delta\to \N$ the function $d([n])=n$ makes $\Delta$ a Reedy category. We have the standard coface maps $d_j:[n]\to [n+1]$, $j=0,\ldots, n+1$, and codegeneracy maps $s_i:[n]\to [n-1]$, $i=0,\ldots, n-1$.

 Let $\Spc$ denote the category of simplicial sets, $\Spc_\bullet$ the category of pointed simplicial sets, each with the standard model structures, see \cite[\S 3.2]{Hovey}.  Note that this is not the Reedy model structure!

Let  $\Delta[n]$ be the representable simplicial set, 
$\Delta[n]:=\Hom_{\Delta}(-,[n])$ and let $\Delta[*]:\Delta\to \Spc$ be the cosimplicial space $n\mapsto \Delta[n]$.
\end{ex}

\subsection{Simplicial structure} We consider a small category $\sS$ and a simplicial model category $\sM$ satisfying the conditions discussed in the previous section. Both model structures for $\sM^\sS$ discussed above yield simplicial model categories: for a simplicial set $A$ and a functor $\sX:\sS\to \sM$, the product $\sX\otimes A$ and Hom-object $\sHom(A, \sX)$ are the evident functors $(\sX\otimes A)(s):=\sX(s)\otimes A$ and $\sHom(A,\sX)(s):=\sHom(A, \sX(s))$. The simplicial Hom-object $\Map_{\sM^\sS}(\sX,\sY)$ is given as the simplicial set 
\[
n\mapsto \Hom_{\sM^\sS}(\sX\otimes\Delta[n], \sY), 
\]
or equivalently, as the equalizer
\[
\Map_{\sM^\sS}(\sX,\sY)\to \prod_{s\in \sS}\Map_{\sM}(\sX(s), \sY(s))\xymatrix{\ar@<3pt>[r]^{\prod g^*}\ar@<-3pt>[r]_{\prod g_*}&}\prod_{g:s\to s'}\Map_{\sM}(\sX(s), \sY(s')).
\]
Together with the evident adjunction $\Hom(\sX, \sHom(A, \sY))\cong \Hom(\sX\otimes A, \sY)$, this makes $\sM^\sS$ into a simplicial model category (see below).

We have as well the object $\sHom(\sA, X)$ in $\sM^\sS$ for $\sA\in \Spc^{\sS^\op}$, $X\in \sM$, with $\sHom(\sA, X)(s):=\sHom(\sA(s), X)$ and the object $X\otimes\sA$  in $\sM^\sS$ for $\sA\in \Spc^{\sS}$, $X\in \sM$, with $(X\otimes\sA)(s):=X\otimes\sA(s)$.

For $\sA\in \Spc^\sS$, $\sX\in \sM^\sS$, we have  $\sHom^\sS(\sA,\sX)$ in $\sM$ defined as the  equalizer
\[
\sHom^\sS(\sA,\sX)\to \prod_{s\in \sS}\sHom(\sA(s), \sX(s))\xymatrix{\ar@<3pt>[r]^{\prod g^*}\ar@<-3pt>[r]_{\prod g_*}&}\prod_{g:s\to s'}\sHom(\sA(s), \sX(s')).
\]
Similarly, for $\sA\in \Spc^{\sS^\op}$, $\sX\in \sM^\sS$, we have   $\sX\otimes^\sS\sA$ in $\sM$, defined as the coequalizer
\[
\amalg_{g:s'\to s}\sX(s')\otimes \sA(s)\xymatrixcolsep{55pt}\xymatrix{\ar@<3pt>[r]^{\prod\sX(g)\otimes\id}\ar@<-3pt>[r]_{\prod\id\otimes \sA(g)}&}
\amalg_{s\in \sS}\sX(s)\otimes \sA(s)\to \sX\otimes^\sS \sA.
\]
 
Besides the adjunction already mentioned, one has the adjunctions
\[
\Hom_\sM(X, \sHom^\sS(\sA,\sY))\cong \Hom_{\sM^\sS}(X\otimes \sA, \sY)
\]
for $X\in \sM$, $\sA\in \Spc^\sS$, $\sY\in \sM^\sS$, 
and 
\[
\Hom_{\sM^\sS}(\sX, \sHom(\sA,Y))\cong \Hom_{\sM}(\sX\otimes^\sS \sA, Y)
\]
for $\sX\in \sM^\sS$, $\sA\in \Spc^{\sS^\op}$, $Y\in \sM$. These all follow directly from the adjunctions for $\sHom$ and $\otimes$. 

Both adjunctions are Quillen adjunctions of two variables. In case $\sM$ is cofibrantly generated and we use the projective model structure, this is \cite[theorem 11.7.3]{Hirschhorn}; if $\sS$ is a Reedy category and we give $\sM^\sS$ the Reedy model structure, this is \cite[lemma 3.24]{Barwick}.
This gives $\sM^\sS$ the structure of a $\Spc^\sS$ model category and a $\Spc^{\sS^\op}$ model category.

\subsection{Monoidal structure}
We now suppose that $\sM$ has a symmetric monoidal structure $\otimes_\sM$, making $\sM$ into a closed symmetric monoidal simplicial model category, with internal Hom $\sHom_\sM(-,-)$.

For $X\in \sM$, $\sY\in \sM^\sS$, we have $X\otimes_\sM \sY$ and $\sHom_\sM(X, \sY)$  in $\sM^\sS$, defined objectwise, with the adjunction, for $\sY, \sZ\in \sM^\sS$, $X\in\sM$,
\[
\Hom_{\sM^\sS}(X\otimes_\sM \sY, \sZ)\cong \Hom_{\sM^\sS}(\sY ,\sHom_\sM(X, \sZ))
\]
This extends to the adjunction on mapping spaces
\[
\Map_{\sM^\sS}(X\otimes_\sM \sY, \sZ)\cong \Map_{\sM^\sS}(\sY ,\sHom_\sM(X, \sZ)).
\]

We define the $\sM$-valued internal Hom
\[
\sHom_\sM^\sS:(\sM^\sS)^\op\times\sM^\sS\to \sM
\]
as the equalizer
\[
\sHom_\sM^\sS(\sX, \sY)\to \prod_{s\in \sS}\sHom_\sM(\sX(s), \sY(s)) \xymatrix{\ar@<3pt>[r]^{\prod g^*}\ar@<-3pt>[r]_{\prod g_*}&}\prod_{g:s\to s'}\sHom_\sM(\sX(s), \sY(s')).
\]
Similarly, for $\sX\in \sM^{\sS}$, $\sY\in \sM^{\sS^\op}$, we have   $\sX\otimes^\sS_\sM\sY$ in $\sM$, defined as the coequalizer
\[
\amalg_{g:s'\to s}\sX(s')\otimes_\sM \sY(s)\xymatrixcolsep{55pt}\xymatrix{\ar@<3pt>[r]^{\prod\sX(g)\otimes\id}\ar@<-3pt>[r]_{\prod\id\otimes \sY(g)}&}
\amalg_{s\in \sS}\sX(s)\otimes_\sM \sY(s)\to \sX\otimes^\sS_\sM \sA.
\]

We have the adjunctions, for $\sA\in \Spc^\sS$, $\sY\in \sM^\sS$, $X\in \sM$,
\[
\sHom^\sS(\sA,\sHom_\sM(X, \sY))\cong \sHom^\sS_\sM(X\otimes\sA, \sY)\cong \sHom_\sM(X, \sHom^\sS(\sA, \sY)),
\]
induced by the adjunctions 
\[
\sHom(A,\sHom_\sM(X, Y))\cong \sHom_\sM(X\otimes A, Y)\cong \sHom_\sM(X, \sHom(A, Y))
\]
for $X, Y\in\sM$, $A\in \Spc$. Analogous constructions and statements hold in the pointed setting.

\begin{lem}\label{lem:Adj2} Give $\sM^\sS$ either the Reedy model structure or, in case $\sM$ is cofibrantly generated,  the projective model structure. Then the operations $\otimes_\sM$ and $\sHom_\sM^\sS$ form a Quillen adjunction of two variables, that is, these make $\sM^\sS$ into an $\sM$-model category.
\end{lem}

\begin{proof}  For the projective model structure, the proof of \cite[theorem 11.7.3]{Hirschhorn} extends word for word to prove the result; the case of the Reedy model structure is proven in \cite[lemma 3.36]{Barwick}
\end{proof}

$\sM^\sS$ is a closed symmetric monoidal category, with $(\sA\otimes_{\sM^\sS}\sB)(s):=\sA(s)\otimes_\sM\sB(s)$ for $\sA, \sB\in \sM^\sS$. The internal Hom is given as 
\[
\sHom_{\sM^\sS}(\sA, \sB)(s) :=\sHom^{s/\sS}(s/\sA, s/\sB)
\]
where $s/\sA\in \sM^{s/\sS}$ is the functor $s/\sA(s\to t):=\sA(t)$; for $f:s\to s'$, the induced map $\sHom_{\sM^\sS}(\sA,\sB)(s)\to \sHom_{\sM^\sS}(\sA,\sB)(s')$ is the map 
\[
\sHom^{s/\sS}(s/\sA, s/\sB)\to \sHom^{s'/\sS}(s'/\sA, s'/\sB)
\]
induced by the functor $f^*:s/\sS\to s'/\sS$, noting that $(s'/\sA)\circ f^*=s/\sA$. The unit is the constant functor with value the unit in $\sM$. 

\begin{rem} The question of when this gives $\sM^\sS$ the structure of a symmetric monoidal model category does not appear to have a simple answer, and we will not need this structure here. In the case of the Reedy model structure, Barwick proves the following result:

\begin{prop}[\hbox{\cite[theorem 3.51]{Barwick}}] Let $\sS$ be a Reedy category and give $\sM^\sS$ the Reedy model structure. Suppose that either\\
a. all morphisms in $\sS_-$ are epimorphisms and for each $s\in \sS$ the category $\sS_-^s$ is either empty or connected \\
or the dual\\
b.  all morphisms in $\sS_+$ are monomorphisms and for each $s\in \sS$, the category $\sS_+^s$ is either empty or  connected.\\
Then $\otimes_{\sM^\sS}$ and $\sHom_{\sM^\sS}(-,-)$ is a Quillen adjunction of two variables, making $\sM^\sS$ a symmetric monoidal model category.
\end{prop}
The condition (a) is satisfied for $\sS=\Delta$ and  the dual (b) is satisfied for $\sS=\Delta^\op$, so the categories of cosimplicial or simplicial objects in a symmetric monoidal model category have the structure of a symmetric monoidal model category. As another example of an $\sS$ satisfying (a), one can take for $\sS$ the category associated to a finite poset having a final object, with Reedy structure $\sS=\sS_-$; a finite poset with initial object similarly satisfies (b) if one takes $\sS=\sS_+$.
\end{rem}

\subsection{Bousfield localization}

We suppose that $\sM$ is cellular and right proper. Let $K$ be a set of cofibrant objects in $\sM$. We have the right Bousfield localization $R_K\sM$ with associated functorial cofibrant replacement $Q_K\to \id$ (see \cite[theorem 5.1.1]{Hirschhorn}). Let $K^\sS$ be the set of cofibrant objects $i_s^*a$, $a\in K$, $s\in \sS$, and let $R_{K^\sS}\sM^\sS$ be the right Bousfield localization of $\sM^\sS$ with respect to $K^\sS$  (as noted in remark~\ref{rem:CellularProper}, $\sM^\sS$ inherits cellularity and right properness from $\sM$).

\begin{lem} \label{lem:FunctBous} Suppose that $\sM$ is cellular and right proper, and give $\sM^\sS$ the projective model structure. Let $K$ be a set of cofibrant objects in $\sM$. \\
1. The right Bousfield localization  $R_{K^\sS}\sM^\sS$ is the same as the projective model structure on $(R_K\sM)^\sS$.\\
2.  Take $x\in \sM^\sS$ and let $Qx\to x$ be a cofibrant replacement in $R_{K^\sS}\sM^\sS$. Then $i_{s*}Qx\to i_{s*}x$ is a cofibrant replacement of $i_{s*}x$ in $R_K\sM$ for all $s\in \sS$.
\end{lem}

\begin{proof} Right Bousfield localization leaves the fibrations unchanged, hence $R_{K^\sS}\sM^\sS$   and $(R_K\sM)^\sS$ have the same fibrations. The weak equivalences in a right Bousfield localization with respect to a set of objects $K$ are the $K$-colocal weak equivalences, that is, maps $X\to Y$ that induce a weak equivalence on the Hom spaces $\sHom(a, RX)\to \sHom(a, RY)$ for all $a\in K$, where $RX$, $RY$ are fibrant replacements. From this it follows that $X\to Y$ is a weak equivalence in $R_{K^\sS}\sM^\sS$ if and only if $i_{s*}X\to i_{s*}Y$ is a weak equivalence in $R_K\sM$ for all $s$, that is, the weak equivalences in $R_{K^\sS}\sM^\sS$   and $(R_K\sM)^\sS$ agree. 

(2) follows from (1), noting that $i_{s*}$ preserves cofibrations, fibrations and weak equivalences (for the projective model structure). 
\end{proof}

\begin{exs}\label{exs:FunctorialConstructions}
1. {\em ``Topological'' Postnikov towers.} We recall a functorial construction of the $(n-1)$-connected cover $f_n\sX\to \sX$ of a pointed space. Fix an integer $n\ge0$ and let $K_n$ be the set of spaces of the form $\Sigma^mX$, with $X$ in $\Spc_\bullet$ and $m\ge n$. $\Spc_\bullet$ is a right proper cellular simplicial model category, hence by \cite[theorem 5.1.1]{Hirschhorn}, the right Bousfield localization $R_{K_n}\Spc_\bullet$ of $\Spc_\bullet$ with respect to the $K_n$-colocal maps exists. In addition, there is a cofibrant replacement functor $f_n:R_{K_n}\Spc_\bullet\to R_{K_n}\Spc_\bullet$. By the definition of right Bousfield localization (\cite[definition 3.3.1]{Hirschhorn}, see also \cite[theorem 2.5]{LevineComp})  $f_n\sX\to \sX$  in $\Ho\Spc_\bullet$ is universal for maps from $(n-1)$-connected $\sY$ to $\sX$; by obstruction theory, it follows that $f_n\sX\to \sX$ is an $(n-1)$-connected cover of $\sX$. Using lemma~\ref{lem:FunctBous}, we may form the $(n-1)$-connected cover $f_n^\sS\sX\to \sX$ in the functor category $\Spc_\bullet^\sS$ as the cofibrant replacement with respect to the right Bousfield localization $R_{K_n^\sS}\Spc_\bullet$.  

Varying $n$ and noting that $K_n\subset K_m$ if $n\ge m$ gives the tower of cofibrant replacement functors
\[
\ldots\to f_{n+1}^\sS\to f_n^\sS\to\ldots\to f_0^\sS=\id.
\]

Let $\Spt$ be the category of $S^1$-spectra in $\Spc_\bullet$, with stable model structure as defined in \cite{HoveySpectra}. We have the $n$th evaluation functor $ev_n:\Spt\to \Spc_\bullet$, $ev_n(S_0, S_1,\ldots):=S_n$, and its left adjoint $F_n:\Spc_\bullet\to \Spt$, 
\[
F_n(S):=(pt,\ldots,pt, S, \Sigma S, \Sigma^2 S,\ldots). 
\]
We repeat the construction of the Postnikov tower, with $\Spt$ replacing $\Spc_\bullet$ and taking $K_n$ to be  the set of objects $F_a\Sigma^bX$, with $X\in\Spc_\bullet$, $b-a\ge n$, $n\in \ZZ$. This gives us the Postnikov tower in the functor category $\Spt^\sS$ (with $n\in\ZZ$)
\[
\ldots\to f_{n+1}^\sS\to f_n^\sS\to\ldots\to  \id.
\]

We may extend these constructions to other model categories. Rather than attempting an axiomatic discussion, we content ourselves with the examples arising in motivic homotopy theory. Let  $S$ be a noetherian separated base scheme and let $\Spc_\bullet(S)$ be the category of pointed spaces over $S$, that is, $\Spc_\bullet$-valued presheaves on the category $\Sm/S$ of smooth $S$-schemes of finite type. We give  $\Spc_\bullet(S)$ the motivic model structure; this gives $\Spc_\bullet(S)$ the structure of a proper combinatorial symmetric monoidal simplicial model category (for details see  \cite[corollary 1.6]{Hornbostel}, \cite[\S1, theorem 1.1]{Jardine}, \cite[Appendix A]{Jardine2} and \cite[theorem 2.3.2]{Pelaez}).  Letting $K_n(S)$ be the set of objects of the form $\Sigma^m\sX$, with $\sX\in \Spc_\bullet(S)$ and $m\ge n$, we have the right Bousfield localization, $R_{K_n(S)}\Spc_\bullet(S)$ and the cofibrant replacement functor $f_n$, with  universal property for maps with source in  the $K_n(S)$-cellular objects of $\Spc_\bullet(S)$. In case $S=\Spec k$, $k$ a field, these turn out to be the $(n-1)$-connected objects in $\Spc_\bullet(S)$, that is, those objects with vanishing $\A^1$-homotopy sheaves $\pi_m$ for $m<n$ (see e.g.  \cite{PelaezUnstable}. See also \cite[theorem 3.1, remark 3.3]{LevineComp} for a discussion of the stable case and an indication of how this construction works in the unstable case).

We may also use categories of $S^1$ or $\P^1$ spectra, $\Spt_{S^1}(S)$, $\Spt_{\P^1}(S)$, with the respective motivic model structures (see \cite{Jardine2} for a description of the model structures and e.g. \cite[theorem 2.5.4]{Pelaez} for the fact that these are cellular). For $S^1$ spectra, replace $K_n$ with $K_n^{S^1}(S):=\{ F_q^{S^1}\Sigma^p\sX,  \sX\in \Spc_\bullet(S), p-q\ge n\}$. Here $F^{S^1}_q:\Spc_\bullet(S)\to \Spt_{S^1}(S)$ is given by using the functor $F_q:\Spc_\bullet\to \Spt$, that is,
\[
F^{S^1}_q(\sX)(T):=F_q(\sX(T))
\]
for each $T\to S$ in $\Sm/S$. Suppose $S=\Spec k$, with $k$ a perfect field. Again, the $K^{S^1}_n(S)$-cellular objects are those $E\in \Spt_{S^1}(S)$ with stable $\A^1$-homotopy sheaves $\pi_m E$ that vanish for $m<n$ and in this stable model category, the  subcategory 
\[
\SH_{S^1}(S)^{\le 0}:=\Ho R_{K_0^{S^1}(S)}\Spt_{S^1}(S) 
\]
of the homotopy category $\SH_{S^1}(S)$ of $\Spt_{S^1}(S)$ is half of a $t$-structure with heart the strictly $\A^1$-invariant Nisnevich sheaves on $\Sm/S$,  and with $\SH_{S^1}(S)^{\ge 0}$ the full subcategory of the $E$ with $\pi_nE=0$ for $n>0$.  This all follows from results of Morel,  see \cite[theorem  4.3.4, lemma 4.3.7]{MorelICTP}.

For $\Spt_{\P^1}(S)$, we use  $K_n^{\P^1}(S):=\{ F_q^{\P^1}\Sigma^p_{S^1}\sX,  \sX\in \Spc_\bullet(S), p-q\ge n\}$, with $F_q^{\P^1}\sX:=(F_q^{\P^1}\sX_0, F_q^{\P^1}\sX_1,\ldots)$, $F^{\P^1}_q\sX_n=\pt$ for $n<q$, $F_q^{\P^1}\sX_n=\Sigma_{\P^1}^{n-q}\sX$ for $n\ge q$, and with identity bonding maps. Assuming that $S=\Spec k$ with $k$ a perfect field, the $K^{\P^1}_n(S)$-cellular objects are those $\sE\in \Spt_{\P^1}(S)$ with stable $\A^1$-homotopy sheaves $\pi_{m+q,q}\sE$ zero for $m<n$, $q\in \ZZ$.  then in this stable model category, the  subcategory $\SH(S)^{\le 0}:=\Ho R_{K_0^{\P^1}(S)}\Spt_{\P^1}(S)$ of the homotopy category $\SH(S)$ of $\Spt_{\P^1}(S)$ is half of a $t$-structure  with $\SH(S)^{\ge 0}$ the full subcategory of the $\sE$ with $\pi_{m+q,q}E=0$ for $m>0$, $q\in\ZZ$.   The heart is Morel's category of ``homotopy modules'' \cite[definition 5.2.4]{MorelICTP}, see \cite[theorem 5.2.3, theorem 5.2.6]{MorelICTP} for detailed statements.
\\\\
2. {\em Slice towers.} This is modification of the construction in $\Spc_\bullet(S)$ given  in (1), using the set $K_n^t$ of objects of the form $\Sigma^b_{\G_m}\sX$, with $b\ge n\ge0$. The $S^1$-stable version uses the set  of objects of the form 
$F_m\Sigma^b_{\G_m}\sX$ with $b\ge n\ge0$ and the $\P^1$-stable version uses the set  of objects of the form 
$F^{\P^1}_m\Sigma^b_{\G_m}\sX$ with $b-m\ge n$, $n\in\ZZ$. Varying $n$, the  first two yield the slice tower
\[
\ldots\to f_{n+1}^t\sX\to f_n^t\sX\to\ldots\to f_0^t\sX=\sX
\]
while the $\P^1$-version gives us the doubly infinite tower
\[
\ldots\to f_{n+1}^t\sE\to f_n^t\sE\to\ldots\to \sE.
\]
Replacing $K_n^t$ with $K_n^{t,\sS}$ gives the slice towers
\begin{gather*}
\ldots\to f_{n+1}^{t,\sS}\sX\to f_n^{t,\sS}\sX\to\ldots\to f_0^{t,\sS}\sX=\sX,\\
\ldots\to f_{n+1}^{t,\sS}\sE\to f_n^{t,\sS}\sE\to\ldots\to \sE.
\end{gather*}
in $\Spc_\bullet(S)^\sS$, $\Spt_{S^1}(S)^\sS$ and $\Spt_{\P^1}(S)^\sS$. There are similarly defined versions in categories of $T$-spectra ($T=\A^1/\A^1\setminus\{0\}$) or the various flavors of symmetric spectra. As above, we refer the reader to \cite{PelaezUnstable} and  \cite[theorem 3.1, remark 3.3]{LevineComp} for details. 
\\\\
3. {\em Betti realizations}. Betti realizations are left derived functors  of a left Quillen functor $\An^*$, either on categories of spaces over $k$, or on the various spectrum categories, where $\An^*$ is a left Kan extension of the functor sending a smooth $k$-scheme $X$ to the topological space of its $\CC$-points (with respect to a fixed embedding $k\hookrightarrow\CC$) or if one prefers $\Spc$ as target category, to the singular complex of this space. As a left derived functor of a left Quillen functor, the resulting Betti realization functor on the appropriate homotopy category is constructed by applying $\An^*$ (or some allied construction, in the case of spectra) to a cofibrant resolution for a suitable (cellular) model structure. Thus, we may form a Betti resolution for functor categories by first noting that $\An^*$ extends by applying it pointwise to a left Quillen functor between functor categories, and then applying this to cofibrant resolutions in the domain functor category.

Fix an embedding $\sigma:k\to \CC$. We use the Betti realization of Panin-Pimenov-R\"ondigs \cite{PPR}, modified to pass to $\Spc$ instead of locally compact Hausdorff spaces. This functor arises from the left Quillen functor
\[
\An^*:\Spc_\bullet(k)\to \Spc_\bullet
\]
which is the Kan extension of the functor sending $X\in \Sm/k$ to the singular complex of $X^\an$, this latter being the topological space of $\CC$-points of $X^\sigma$, endowed with the classical topology. 

One extends this to $\P^1$-spectra using the fact that $(\P^1)^\an\cong S^2$, and that $\An^*$ is symmetric monoidal, then using an equivalence of $\Spt$ and $S^2$-spectra. Glossing over this latter equivalence, we have the isomorphism (in $\SH$)
\[
\Re_B(\MGL)\cong \MU.
\]
There is a similar version from symmetric $\P^1$-spectra to symmetric $S^2$-spectra, inducing an isomorphic functor on the 
homotopy categories.

Finally, the Betti realization functor extends to a left Quillen functor
\[
\An^{\sS,*}:\Spt_T(k)^\sS\to \Spt_{S^2}^\sS,
\]
with a natural isomorphism $i_s^*\circ \An^{\sS,*}\cong \An^*\circ i_{s*}$; note that one needs to use a different model structure on $\Spt_T(k)$ than the one we have been using, see \cite[\S A4]{PPR}  and \cite[\S 5]{LevineComp} for details.\footnote{We still use the projective model structure on $\Spt_T(k)^\sS$, but with respect to the Panin-Pimenov-R\"ondigs model structure on $\Spt_T(k)$.}  For other versions of the Betti realization, see  \cite[definition 2.1]{Ayoub},  \cite{Riou} and   \cite[\S 4]{VoevMotivic}.

 We let
\[
\Re_B^\sS:\Ho\Spt_T(k)^\sS\to \Ho\Spt^\sS
\]
be the left derived functor of $\An^{\sS,*}$ composed with the equivalence $\Ho\Spt_{S^2}^\sS\cong \Ho\Spt^\sS$. 
\end{exs}

\begin{rem}  Since the Postnikov tower
\[
\ldots\to f_{n+1}^\sS\to f_n^\sS\to\ldots \to \id
\]
and the slice tower
\[
\ldots\to f_{n+1}^{t,\sS}\to f_n^{t,\sS}\to\ldots\to \id,
\]
are both defined via cofibrant replacement functors on the appropriate model categories, we thus have for $a\le b$ well-defined endofunctors $\tilde{f}_{a/b}^\sS:=\hofib(f_b^\sS\to f_a^\sS)$ and $\tilde{f}_{a/b}^{t,\sS}:=\hofib(f_b^{t,\sS}\to f_a^{t,\sS})$, giving rise to homotopy fiber sequences
\[
\tilde{f}_{a/b}^\sS\to f_b^\sS\to f_a^\sS;\ \tilde{f}_{a/b}^{t,\sS}\to f_b^{t,\sS}\to f_a^{t,\sS}.
\]
In the stable setting, this gives us the homotopy fiber sequences 
\[
f_b^\sS\to f_a^\sS\to f_{a/b}^\sS\text{\ \ and\ \ }\ f_b^{t,\sS}\to f_a^{t,\sS}\to f_{a/b}^{t,\sS}
\]
by defining
\[
f_{a/b}^\sS =\Sigma\circ \tilde{f}_{a/b}^\sS\text{\ \ and\ \ }f_{a/b}^{t,\sS}=\Sigma\circ \tilde{f}_{a/b}^{t,\sS}. 
\]
Even in the unstable setting, if we fix an object $\sY$ and let $\sX=\Omega(\sY)$, then we have a canonical weak equivalence
\[
f_n^\sS(\sX)\cong \Omega(f_{n+1}^\sS(\sY))
\]
giving us the homotopy fiber sequence $f_b^\sS(\sX)\to f_a^\sS(\sX)\to\tilde{f}_{a+1/b+1}^\sS(\sY)$. We may therefore define $f_{a/b}^\sS(\sX):=\tilde{f}_{a+1/b+1}^\sS(\sY)$, giving the homotopy fiber sequence
\[
f_b^\sS(\sX)\to f_a^\sS(\sX)\to f_{a/b}^\sS(\sX)
\]
which is at least functorial in $\sY$ (for $\sX=\Omega(\sY)$). We can make a similar definition of $f_{a/b}^{t,\sS}(\sX)$ in the setting of the slice tower, having chosen a delooping $\sY$ of $\sX$. We set 
\[
s^\sS_n:=f_{n/n+1}^\sS\text{\ \ and\ \ }s^{t,\sS}_n:=f_{n/n+1}^{t,\sS}.
\]

As the Postnikov tower and slice tower are defined via cofibrant replacement functors, the functors $f_n^\sS$ and $f_n^{t, \sS}$ send fibrant objects to fibrant objects. The same holds for the layers $ f_{a/b}^\sS$ and $f_{a/b}^{t,\sS}$, assuming that the choice of delooping is fibrant.
\end{rem}

We conclude with a simple result concerning connectivity.

\begin{lem}\label{lem:Connectivity} Take $\sE\in \Spt^\sS$ such that $i_{s*}\sE$ is $(n-1)$-connected for each $s\in \sS$. Then $f^\sS_n\sE\to \sE$ is a weak equivalence.
\end{lem}

\begin{proof}
Since $i_{s*}f^\sS_n\sE\cong f_ni_{s*}\sE$, our assumption on $\sE$ implies that $i_{s*}f^\sS_n\sE\to i_{s*}\sE$ is a weak equivalence for each $s\in \sS$, and thus $f^\sS_n\sE\to \sE$ is a weak equivalence.
\end{proof}

\section{Cosimplicial objects in a model category}\label{sec:Cosimpl} We will work in a fairly general setting, letting $\sM$ be  a pointed closed  symmetric monoidal  simplicial model category. The reader can keep in mind the example $\sM=\Spc_\bullet$,  the category of   pointed simplicial sets.  

This material, as well as much of the material in the next section, may be found in the beginning of \cite{Bousfield}.

We have the functor category  $\sM^\Delta$ of cosimplicial objects in $\sM$.  We give $\sM^\Delta$ the Reedy model structure, unless stated explicitly otherwise.  For $\sX:\Delta\to \sM$, we often write $\sX^n$ for $\sX([n])$.  
 
\begin{rem} The unit for the monoidal structure on $\sM^\Delta$ is the constant cosimplicial object on the unit $\1$ in $\sM$; this is usually not a cofibrant object in $\sM^\Delta$.

If $A$ is an object in $\sM$, write $cA$ for the constant cosimplicial object. The functor $c$ does not in general preserve cofibrations; however, if $i:A\to B$ is a cofibration in $\sM$ and $p:\sX\to \sY$ is a fibration in $\sM^\Delta$ with $\sY$ (and hence $\sX$) fibrant, then
\[
\sHom(cB, \sX)\to \sHom(cA,\sX)\times_{\sHom(cA, \sY)}\sHom(cB, \sY)
\]
is a fibration, and is a trivial fibration if either $i$ or $p$ is a weak equivalence.
\end{rem}

We consider the full subcategory $\Delta^{\le n}$ of $\Delta$, with objects $[k]$, $k=0,\ldots, n$; note that  $\Delta^{\le n}$ is also a Reedy category with the evident $+$ and $-$ subcategories. We usually give $\sM^{\Delta^{\le n}}$ the Reedy model category structure.

For $T\in\sM$, we write $\Omega_T$ for the functor $\sHom(T,-):\sM\to \sM$, right adjoint to $\Sigma_T$, where $\Sigma_T(X)=X\wedge T$. We also write $\Omega_T$ for the functor $\sHom(cT,-):\sM^\Delta\to \sM^\Delta$, leaving it to the context to determine the precise meaning.  Similarly, we may use the $\Spc_\bullet$-structure to define $\Omega_K:=\sHom(K,-):\sM\to \sM$, right adjoint to $\Sigma_K$, $\Sigma_K(X)=X\wedge K$, and also 
$\Omega_K:=\sHom(K,-):\sM^\Delta\to \sM^\Delta$. We write $\Omega$ and $\Sigma$ for $\Omega_{S^1}$ and $\Sigma_{S^1}$.

\subsection{The total complex and associated towers}
We recall the construction of towers associated to cosimplicial objects, recapping the construction of \cite{BK} for cosimplicial spaces, which was generalized to cosimplicial objects in a simplicial model category in \cite{Bousfield}.

Let $\sX$ be a cosimplicial object in $\sM$. We have the associated total object $\Tot\,\sX:=\sHom^\Delta(\Delta[*], \sX)$ in $\sM$; note that $\Delta[*]$ is a cofibrant object in $\Spc^\Delta$, hence the functor $\Tot:\sM^\Delta\to \sM$ is a right Quillen functor with left adjoint $\sA\mapsto \sA\times\Delta[*]$. We make the analogous definition in the pointed setting.

For $T\in \sM$,  $\sX\in \sM^\Delta$, the adjoint property of $\sHom$ gives the isomorphism in $\sM$
\begin{multline*}
\sHom_{\sM}(T, \Tot\,\sX)\cong \sHom_{\sM}^\Delta(T\times\Delta[*],  \sX)\\
\cong \sHom^\Delta(\Delta[*], \sHom_\sM(T, \sX))=\Tot( \sHom_\sM(T, \sX)).
\end{multline*}
That is,  we have the canonical isomorphism $\Omega_T\Tot\, \sX\cong \Tot\,\Omega_T\sX$. Consequently, assume $\sM$ is a category of $T$-spectra in some model category $\sM_0$. 
For $\sE:\Delta\to\Spt_T^{\sM_0}$ a cosimplicial  $T$-spectrum, 
\[
\sE=(\sE_0,\ldots,\sE_n,\ldots),
\]
with  bonding maps $\epsilon_n:\sE_n\to \Omega_T\sE_{n+1}$, $\Tot\,\sE$ is the spectrum $(\Tot\, \sE_0, \Tot\,\sE_1,\ldots)$ with bonding maps $\Tot\,\epsilon_n:\Tot\,\sE_n\to \Tot\,\Omega_T\sE_{n+1}\cong \Omega_T\Tot\,\sE_{n+1}$.

Let  $i_k:\Delta^{\le k}\to \Delta$ be the inclusion functor, and let $\Spc^{(k)}$ be the category of presheaves of sets on $\Delta^{\le k}$. Restricting via $i_k$ gives the functor $i_{k*}:\Spc\to\Spc^{(k)}$, which admits the left adjoint $i_k^*:\Spc^{(k)}\to \Spc$; the $k$-skeleton functor $\sk_k$ is the composition $i_k^*\circ i_{k*}$ with counit $\sk_k\to\id$. We write $A^{(k)}$ for $\sk_kA$. We have the canonical natural transformations $\iota_{m,k}:\sk_k\to \sk_m$, for $0\le k\le m$, with $\iota_{n,m}\circ \iota_{m,k}=\iota_{n,k}$ for $k\le m\le n$. 

Let $\iota_k:\Delta[*]^{(k)}\to \Delta[*]$ be the $k$-skeleton of $\Delta[*]$, that is, the cosimplicial simplicial set $n\mapsto \Delta[n]^{(k)}$. For $\sX$ a cosimplicial object of $\sM$, let $\Tot_{(k)}\sX:=\sHom^\sM(\Delta[*]^{(k)}, \sX)$.  The sequence of cofibrations
\[
\0:=\Delta^{(-1)}\hookrightarrow \Delta[*]^{(0)}\hookrightarrow\Delta[*]^{(1)}\hookrightarrow\ldots\hookrightarrow\Delta[*]^{(k)}\hookrightarrow\ldots\Delta[*]
\]
thus gives the tower in $\sM$
\begin{equation}\label{eqn:Tower}
\Tot\,\sX\to\ldots\to \Tot_{(k)}\sX\to \ldots\to \Tot_{(1)}\sX\to \Tot_{(0)}\sX\to \Tot_{(-1)}\sX:=\pt
\end{equation}
which is a tower of fibrations if $\sX$ is fibrant.  

We let $\Tot^{(k)}\sX\to \Tot\,\sX$ be the homotopy fiber of $\Tot\,\sX\to \Tot_{(k-1)}\sX$, giving the tower in $\sM$
\begin{equation}\label{eqn:Tower2}
\ldots\to\Tot^{(k+1)}\sX\to \Tot^{(k)}\sX\to\ldots\to \Tot^{(1)}\sX\to \Tot^{(0)}\sX= \Tot\,\sX.
\end{equation}

For $m\ge k\ge-1$, let $\Tot_{(m/k)}\sX$ be the homotopy fiber of 
$\Tot_{(m)}\sX\to \Tot_{(k)}\sX$.  For $m\ge k\ge0$, let $\Tot^{(k/m)}\sX=\Tot_{(m-1/k-1)}\sX$. To unify the notation, we define 
\[
\Tot^{(0)}:=\Tot=:\Tot_{(\infty)}\text{\ \ and\ \ } \Tot^{(m/\infty)}:=\Tot^{(m)}=:\Tot_{(\infty/m-1)}.
\]

The homotopy fiber sequence $\Tot_{(m-1/k-1)}\sX\to \Tot_{(m-1)}\sX\to \Tot_{(k-1)}\sX$ and an application of the Quetzalcoatl lemma  to the commutative diagram
\[
\xymatrix{
\Tot\,\sX\ar@{=}[d]\ar[r]&\Tot_{(m-1)}\sX\ar[d]\\
\Tot\,\sX\ar[r]&\Tot_{(k-1)}\sX
}
\]
gives us the homotopy fiber sequence 
\begin{equation}\label{eqn:FiberSeq0}
\Omega\Tot^{(k/m)}\sX\to \Tot^{(m)}\sX\to \Tot^{(k)}\sX.
 \end{equation}
 
 Suppose we have a delooping $\sZ$ of $\sX$.  We get natural (in $\sZ$) deloopings 
$\Tot^{(r)}\sX\cong \Omega \Tot^{(r)}\sZ$
 for all $r$, which in turn give us natural deloopings $\Tot^{(r/s)}\sX\cong \Omega \Tot^{(r/s)}\sZ$ for all $r, s$. Extending the homotopy fiber sequence 
 \eqref{eqn:FiberSeq0} for $\sZ$ to the left gives us the homotopy fiber sequence 
\[
\Omega\Tot^{(m)}\sZ\to \Omega\Tot^{(k)}\sZ\to \Omega\Tot^{(k/m)}\sZ; 
 \]
 using our deloopings thus gives us the homotopy fiber sequence
 \begin{equation}\label{eqn:FiberSeq1}
\Tot^{(m)}\sX\to \Tot^{(k)}\sX\xrightarrow{\rho_{k/m}} \Tot^{(k/m)}\sX,
 \end{equation}
 with $\rho_{k/m}$ natural in $\sZ$. In fact, $\rho_{k/m}$  is natural in $\sX$ and may be defined directly as follows: the composition $\Tot^{(k)}\sX\to \Tot\,\sX\to \Tot_{(m-1)}\sX\to \Tot_{(k-1)}\sX$ is just the composition in the homotopy fiber sequence $\Tot^{(k)}\sX\to \Tot\,\sX\to\Tot_{(k-1)}\sX$, so is canonically homotopic to the trivial map. This gives the canonical lifting of the composition $\Tot^{(k)}\sX\to \Tot\,\sX\to \Tot_{(m-1)}\sX$ to a map $\Tot^{(k)}\sX\to  \Tot_{(m-1/k-1)}\sX$ which is just the map $\rho_{k/m}$.

In case $\sM$ is a stable model category, the loops functor $\Omega$ is invertible in the homotopy category, so our delooping assumption is automatically satisfied, and we just define $\Tot^{(k/m)}\sX$ as the homotopy fiber of $\tilde{\Sigma}\Tot_{(m-1)}\sX\to \tilde{\Sigma}\Tot_{(k-1)}\sX$, where $\tilde{\Sigma}$ is the functorial fibrant model of the suspension.

We fix a homotopy functor $\pi_*$ on $\sM$.  Rather than try to give an axiomatic treatment, we list the examples of interest:
\begin{enumerate} 
\item $\sM=\Spc_\bullet$, $\pi_*$ the usual direct sum of the homotopy groups (pointed sets for $*=0$).
\item $\sM=\Spc_\bullet(S)$, $\pi_*$ the Nisnevich sheaf of $\A^1$-homotopy groups (pointed sets for $*=0$), and $S=\Spec k$ with $k$ a perfect field.
\item $\sM=\Spc_\bullet(S)$, $\pi_n:=\oplus_{m\ge0}\pi_{n+m,m}$, and $S=\Spec k$ with $k$ a perfect field.
\end{enumerate}
These all have the property that a map $f:X\to Y$ in $\sM$ is a weak equivalence if and only if $f$ induces an isomorphism on $\pi_*$ for all choices of base point in $X$.\footnote{In cases (2) and (3),  the choice of base point is a local one, with respect to the Nisnevich topology.} For the case of a stable model category, we will assume that $\pi_*$ is the graded truncation functor associated to a nondegenerate $t$-structure on $\Ho\sM$ and again that a map $f:X\to Y$ in $\sM$ is a weak equivalence if and only if $f$ induces an isomorphism on $\pi_*$. Our main examples of interest are
\begin{enumerate}
\item $T=S^1$, $\sM=\Spt_{S^1}(S)$, $\pi_n$ the stable $\A^1$-homotopy sheaf.
\item $T=S^1$, $\sM=\Spt_{S^1}(S)$, $\pi_n:=\oplus_{m\ge 0}\pi_{n+m,m}$, $n\in\ZZ$, with $\pi_{a,b}$ the bigraded stable $\A^1$-homotopy sheaf.
\item $T=\P^1, \A^1/\A^1\setminus\{0\}$ or some other convenient model of $\P^1$, $\sM=\Spt_T(S)$, $\pi_n:=\oplus_{m\in\ZZ}\pi_{n+m,m}$, $n\in\ZZ$,
\end{enumerate}

For a cosimplicial abelian group $n\mapsto A^n$, we have the associated complex $A^*$ with differential the alternating sum of the coface maps. We also have the quasi-isomorphic normalized subcomplex $NA^*$ with $NA^n:=\cap_{i=0}^{n-1}\ker s_i$.  For a cosimplicial   object $\sX\in \sM^\Delta$, let $N\sX^n$ be the fiber of $s^n:\sX^n\to M^n(\sX)$ (over the base point).  

\begin{lem}\label{lem:FiberSeq1} There is a natural isomorphism  of $\Omega_{\Delta[n]/\del\Delta[n]}N\sX^n$ with the fiber of the map $\Tot_{(n)}\sX\to \Tot_{(n-1)}\sX$. If $\sX$ is fibrant, the induced map $\Omega_{\Delta[n]/\del\Delta[n]}N\sX^n\to \Tot_{(n/n-1)}\sX$ gives rise to an isomorphism
\begin{equation}\label{eqn:LayerIso}
 \Omega^n N\sX^n\cong \Tot_{(n/n-1)}\sX.
\end{equation}
in $\Ho\sM$. Moreover, we have an isomorphism $\pi_jN\sX^n\cong N(\pi_j\sX)^n\subset \pi_j\sX^n$.
\end{lem}

\begin{proof}  See \cite[Chap. X, Prop. 6.3]{BK} for a proof in the case $\sM=\Spc$. 

The fiber of $\Tot_{(n)}\sX\to \Tot_{(n-1)}\sX$ is equal to 
$\sHom_\sM(\Delta[*]^{(n)}/\Delta[*]^{(n-1)}, \sX)$. This in turn is isomorphic to the equalizer
\[
\sHom_\sM(\Delta[*]^{(n)}/\Delta[*]^{(n-1)}, \sX)\to
\sHom(\Delta[n]/\del\Delta[n], \sX^n)\xymatrix{\ar@<3pt>[r]^\alpha\ar@<-3pt>[r]_\beta&}\prod_{i=0}^{n-1}\sX^{n-1}
\]
where $\alpha(f)=\prod_is_i\circ f$ and $\beta$ is the map to the base point. This gives the asserted identification of  $\sHom_\sM(\Delta[*]^{(n)}/\Delta[*]^{(n-1)}, \sX)$ with  $\Omega_{\Delta[n]/\del\Delta[n]}N\sX^n$.

Assume that $\sX$ is fibrant. As  $\Delta[*]^{(n-1)}\to \Delta[*]^{(n)}$ is a cofibration, the map $\Tot_{(n)}\sX\to \Tot_{(n-1)}\sX$ is a fibration, hence the induced map $\Omega_{\Delta[n]/\del\Delta[n]}N\sX^n\to \Tot_{(n/n-1)}\sX$
 is a weak equivalence. Since $\sX$ is fibrant,  so is $N\sX^n$, hence a weak equivalence $(S^1)^{\wedge n}\to \Delta[n]/\del\Delta[n]$ induces a weak equivalence $\Omega_{\Delta[n]/\del\Delta[n]}N\sX^n\to \Omega^nN\sX^n$. The last assertion is proven for simplicial sets in  \cite[Chap. X, Prop. 6.3]{BK}; the same proof works in general.
\end{proof}

Consider the following conditions on a cosimplicial pointed space $\sX$:
\begin{equation} \label{eqn:ConvCond}
\vbox{\begin{enumerate}
\item $\sX$ is fibrant and there is a fibrant cosimplicial object $\sY$ in $\sM^\Delta$ and an isomorphism $\sX\cong \Omega^2\sY$ in $\Ho\sM^\Delta$.
\item Given an integer $i\ge0$,  there is an integer $N_i$ such that $(N\pi_j\sX)^n=0$ for $n\ge N_i$, $j\le i+n$. 
\end{enumerate}}
\end{equation}
In the stable case,  we have the analog of these conditions for $\sX\in \sM^\Delta$, namely,
\begin{equation} \label{eqn:SpectrumConvCond}
\vbox{\begin{enumerate}
\item $\sX$ is  fibrant.
\item Given an integer $i$,  there is an integer $N_i$ such that $(N\pi_j\sX)^n=0$ for $n\ge N_i$, $j\le i+n$.
\end{enumerate}}
\end{equation}

By lemma~\ref{lem:FiberSeq1},  under the assumption that $\sX$ is fibrant, the condition
\eqref{eqn:ConvCond}(2) is equivalent to
\[
\pi_j\Tot_{(n/n-1)}\sX=0\text{ for } j\le i, n\ge N_i.
\]

\subsection{Spectral sequences and convergence}\label{subsec:TotSSeq}
Suppose that $\sX\in\sM_\bullet^\Delta$ is fibrant. The tower of fibrations \eqref{eqn:Tower} gives the  spectral sequence
\begin{equation}\label{eqn:TowerSS}
{}_*E_1^{p,q}(\sX)=\pi_{-p-q}\Tot_{(p/p-1)}\sX\Longrightarrow \pi_{-p-q}\Tot_{(B-1/A-1)}\sX,\ A\le p< B,
\end{equation}
for $0\le A< B\le \infty$. 
Note that we use the Cartan-Eilenberg indexing convention instead of the Bousfield-Kan convention used in \cite{BK}. 

Under the assumption \eqref{eqn:ConvCond}(1)  or \eqref{eqn:SpectrumConvCond}(1), the spectral sequence \eqref{eqn:TowerSS} is  isomorphic to the spectral sequences of the tower \eqref{eqn:Tower2}:

\begin{equation}  \label{eqn:TowerSS2}
E_1^{p,q}(\sX)=\pi_{-p-q}\Tot^{(p/p+1)}\sX\Longrightarrow \pi_{-p-q}\Tot^{(A/B)}\sX,\  A\le p<B,
\end{equation}
for $0\le A<B\le\infty$. 
Furthermore, using lemma~\ref{lem:FiberSeq1} and \eqref{eqn:LayerIso}, the $E_1$-terms are
\[
E_1^{p,q}(\sX)=N\pi_{-q}\sX^p.
\]

\begin{lem}\label{lem:Convergence}  
1. If   $\sX\in \sM_\bullet^\Delta$ satisfies \eqref{eqn:ConvCond}(1), (or \eqref{eqn:SpectrumConvCond}(1) if $\sM$ is a stable model category), then the spectral sequences \eqref{eqn:TowerSS} and \eqref{eqn:TowerSS2} are strongly convergent if $B<\infty$.\\
2. If  $\sX \in \sM_\bullet^\Delta$ satisfies  \eqref{eqn:ConvCond}  (or \eqref{eqn:SpectrumConvCond} if $\sM$ is a stable model category). Then the spectral sequences \eqref{eqn:TowerSS} and  \eqref{eqn:TowerSS2}  are  strongly convergent for all $B$, including $B=\infty$. \\
\end{lem}

\begin{proof} It suffices to give the proof in the unstable case. (1) follows easily, as in all cases, the associated tower is finite.

For (2), since $\sX\cong \Omega^2\sY$, there are no low-dimensional subtleties, and all the statements we will be using from \cite{BK} make sense and are valid for $\pi_1$ and $\pi_0$. 

 We first show that the sequence is bounded.  Indeed, $E_1^{p,q}=0$ for $p<0$, and if $p+q=-n$, then $E_1^{p, q}=0$ for $p\ge N_n$. In particular, $E_r^{p, -n-p}=E_{r+1}^{p, -n-p}=E_\infty^{p, -n-p}$ for $r\ge \max\{N_n, N_{n-1}\}$ and all $p$.

Thus, the terms $\{E_r\}$ are ``Mittag-Leffler in dimension $i$'' for all $i$ \cite[IX, \S5, pg. 264]{BK} and hence, by \cite[IX, proposition 5.7]{BK}
the spectral sequence converges completely to $\pi_*\Tot\,\sX$. Fix an integer $n\ge0$. Since the sequence is bounded,  the filtration of $\pi_n\Tot\,\sX$ induced by the spectral sequence is finite for each $n$, giving the desired convergence.
\end{proof}

\begin{lem} \label{lem:conn} Suppose  there is an integer $c$ such that $\sX^n$ is $(c-1)$-connected for all $n$. Then for  all $0\le r\le\infty$, $m\in\ZZ$  (in the unstable case, we assume in addition $m\ge0$):\\
1. For $0\le c-m\le r$, the map $\Tot^{(c-m/r)}\sX\to \Tot^{(0/r)}\sX$ induces a surjection 
\[
\pi_m\Tot^{(c-m/r)}\sX\to \pi_m\Tot^{(0/r)}\sX.
\]
2. For $0\le c-m-1\le r$, the map $\Tot^{(c-m-1/r)}\sX\to \Tot^{(0/r)}\sX$ induces an isomorphism 
\[
\pi_m\Tot^{(c-m-1/r)}\sX\to \pi_m\Tot^{(0/r)}\sX.
\]
\end{lem}

\begin{proof} We have the strongly convergent spectral sequence
\[
E_1^{p,q}(\sX)=\pi_{-p-q}\Tot^{(p/p+1)}\sX\Longrightarrow \pi_{-p-q}\Tot^{(0/b)}\sX;\quad 0\le p\le b-1.
\]
By  lemma~\ref{lem:FiberSeq1}, $E_1^{p,q}=\pi_{-q}N\sX^p\subset \pi_{-q}\sX^p$, so $E_1^{p,q}=0$ for $-q<c$. Since  $E_1^{p,q}=0$ for $p> b-1$ this implies that  $E_1^{p,q}=0$ for $-p-q\le c-b$. Thus  $\pi_s\Tot^{(0/b)}\sX=0$ for $s\le c-b$, so 
$\pi_s\Tot^{(0/c-m)}\sX=0$ for $s\le m$. Using the homotopy fiber sequence
\[
\Tot^{(c-m/r)}\sX\to \Tot^{(0/r)}\sX\to \Tot^{(0/c-m)}\sX
\]
proves (1). Similarly,  $\pi_s\Tot^{(0/c-m-1)}\sX=0$ for $s\le m+1$, and (2) follows by a similar argument. 
\end{proof}

\section{Cosimplices and cubes} \label{CosimpCube}
The functors $\Tot_{(n)}$ are complicated by the mixture of codegeneracies and coface maps in $\Delta$; in this section we discuss the reduction of $\Tot_{(n)}$  to a homotopy limit over an associated direct category, namely, a punctured $(n+1)$-cube.

As above, we have   the full subcategory $\Delta^{\le n}$ of $\Delta$ and for a model category $\sC$ the restriction functor $\iota_{n*}:\sC^\Delta\to \sC^{\Delta^{\le n}}$ with left adjoint $\iota_n^*$.

Throughout this section we fix a pointed simplicial model category $\sM$.

\begin{lem} \label{lem:Truncation} Take $\sX$  in $\sM^\Delta$.   There is a natural isomorphism
\[
 \Tot_{(n)}\sX\cong  \sHom(\iota_{n*}\Delta[*], \iota_{n*}\sX);
\]
if $\sX$ is fibrant, there is  a natural weak equivalence
\[
\holim_{\Delta^{\le n}}\iota_{n*}\sX\to \Tot_{(n)}\sX.
\]
\end{lem}

\begin{proof}  We note that we have a canonical isomorphism of cosimplicial spaces
\[
\sk_n\Delta[*]\cong\iota_n^*\iota_{n*}\Delta[*].
\]
Indeed, $(\sk_n\Delta[m])([k])$ is the colimit over $[k]\to[\ell]\in ([k]/\Delta^{\le n})^\op$ of $\Hom_\Delta([\ell],[m])$, while $(\iota_n^*\iota_{n*}\Delta[*])[m]([k])$ is the colimit over $[\ell]\to [m]\in \Delta^{\le n}/[m]$ of $\Hom_\Delta([k],[\ell])$. Both colimits are equal to the subset of $\Hom_\Delta([k],[m])$ consisting of maps  that admit a factorization $[k]\to[\ell]\to[m]$ with $\ell\le n$.

This gives us the isomorphism in $\sM$
\[
\Tot_{(n)}\sX:=\sHom(\sk_n\Delta[*],\sX)\cong \sHom(\iota_{n*}\Delta[*], \iota_{n*}\sX).
\]
For $0\le k\le n$, the nerve of $\Delta^{\le n}_{inj}/[k]$ is the barycentric subdivision of $\Delta[k]$ and sending the nondegenerate $k$-simplex of 
$\Delta[k]$ to the $k$-simplex 
\[
\xymatrix{
\{0\}\ar@{^(->}[r]\ar[drr]&\{0,1\}\ar@{^(->}[r]\ar[dr]&\ldots\ar@{^(->}[r]&\{0,\ldots, k\}\ar[dl]\\
&&\{0,\ldots, k\}&&}
\]
in $\Delta^{\le n}/[k]$ gives rise to an acyclic cofibration 
\[
\alpha:\iota_{n*}\Delta[*]\to [[k]\mapsto \sN\Delta^{\le n}/[k]] 
\] 
 in $\Spc^{\Delta^{\le n}}$.  As $\sX$ is fibrant in $\sM^\Delta$, it follows that $ \iota_{n*}\sX$ is fibrant in $\sM^{\Delta^{\le n}}$, so  $\alpha$ induces a weak equivalence upon applying $\sHom(-, \iota_{n*}\sX)$. As
\[
\holim_{\Delta^{\le n}}\iota_{n*}\sX:=\sHom([[k]\mapsto \sN\Delta^{\le n}/[k]], \iota_{n*}\sX)
\]
by definition,  we thus have the weak equivalences
 \[
 \holim_{\Delta^{\le n}}\iota_{n*}\sX\cong  \sHom(\iota_{n*}\Delta[*], \iota_{n*}\sX)\cong  \Tot_{(n)}\sX.
 \]
 
 \end{proof}
 
 \begin{rem} For $\sM$ the category of pointed simplicial sets, the above result is proven in \cite[Lemma 2.9]{S1}.
 \end{rem}
 
Let $\sq^n$ be the category associated to the set of subsets of $\{1,\ldots, n\}$, with morphisms being inclusions of subsets, and let $\sq^n_0$ the full subcategory of nonempty subsets.We let $i_{I,J}:J\to I$ denote the morphism associated to an inclusion $I\subset J$.  

Give $\{1,\ldots, n\}$ the {\em opposite} of the standard order. The maps $i_{I,J}$  are clearly order preserving, so sending $I$ to the ordered set $[|I|-1]$ by the unique order preserving bijection defines a functor
\begin{equation}\label{eqn:CubeFunctor}
\phi^{n+1}_0:\sq^{n+1}_0\to \Delta^{\le n}\subset \Delta
\end{equation}

For a model category $\sM$ and for $\sC=\sq^{n+1}_0, \sq^n,  \Delta^{\le n}$, we give $\sM^\sC$ the Reedy model structure; as $\sq^{n+1}_0$ and $\sq^n$ are direct categories, this agrees with the projective model structure in these cases.

\begin{prop} \label{prop:SimpVCube} Let $\sM$ be a pointed simplicial model category and
take $\sX$ in $\sM^\Delta$. The map
\[
\phi_0^{n+1*}:\holim_{\Delta^{\le n}}\iota_{n*}\sX\to \holim_{\sq^{n+1}_0} \phi^{n+1}_{0*}\iota_{n*}\sX
\]
induced by the functor $\phi^{n+1}_0:\sq^{n+1}_0\to \Delta^{\le n}$  is a weak equivalence in $\sM$.
\end{prop}

\begin{proof}  As replacing $\sX$ with a fibrant model induces a weak equivalence on the respective homotopy limits, we may assume that $\sX$ is fibrant.  The result in the case $\sM=\sSets_\bullet$  follows from \cite[Theorem 6.5]{S2}. The general case follows from the case of simplicial sets, by taking a cofibrant object $A$ of $\sM$ and applying the mapping space functor $\text{Maps}_\sM(A,-)$ to $\phi_0^{n+1*}$.
\end{proof}

The $n$-cube and punctured $(n+1)$-cube lend themselves to inductive arguments. Take an integer $n\ge1$. We decompose $\sq^{n+1}_0$ into three pieces, by  defining  $\sq^{n-}_0$ to be the full subcategory with objects $I$, $n\not\in I$, 
$\sq^{n+}_0$ the full subcategory with objects $I$, $n\in I$, $I\neq\{n\}$ and $\pt_{n}:=\{n\}$ (with identity morphism). We have the  isomorphisms $j^-_n: \sq^{n}_0\to \sq^{n-}_0$, $j^+_n:\sq^n_0\to  \sq^{n+}_0$ defined as follows: Let $j_n:\{1,\ldots, n\}\to \{1,\ldots, n+1\}$ be the inclusion $j_n(i)=i$ for $1\le i<n$, $j_n(n)=n+1$. Then $j_n^-$ is just the functor induced by $j_n$, and $j_n^+(I)=j_n(I)\cup\{n\}$.  Let $i_n^+:\sq^n_0\to \sq^{n+1}_0$ and $i_n^-:\sq^n_0\to \sq^{n+1}_0$ be the inclusions induced by $j_n^+$ and $j_n^-$.

The inclusions $I\subset I\cup\{n\}$ define a natural transformation $\alpha_n:i_n^-\to i_n^+$, whereas the inclusions $\{n\}\subset I$, $I\in \sq^{n+}_0$, define the morphisms $\beta_I:\{n\}\to i_n^+(I)$. For each $\sX\in \sM^{\sq^{n+1}_0}$, we thus have the  diagram in $\sM$
\[
\xymatrix{
\holim_{\sq^{n}_0}i_{n*}^-\sX\ar[r]^{\alpha_n}&\holim_{\sq^{n}_0}i_{n*}^+\sX\\
&\sX(\hbox{$\{n\}$})\ar[u]_{\beta_*}}
\]
This diagram defines a functor  
\[
{\holim}^{+,-}_{n+1}:\sM^{\sq^{n+1}_0}\to \sM^{\sq^{2}_0}
\]
and we have a natural isomorphism in $\sM$
\[
\holim_{\sq^{n+1}_0}\sX\cong\holim_{\sq^{2}_0}{\holim}^{+,-}_{n+1}(\sX).
\]
In case $\sX(\{n\})=\pt$, we have the natural isomorphisms
\begin{equation}\label{eqn:SqInduction}
\holim_{\sq^{n+1}_0}\sX\cong\holim_{\sq^{2}_0}{\holim}^{+,-}_{n+1}(\sX)\cong \hofib(\alpha_n:\holim_{\sq^{n}_0}i_{n*}^-\sX\to
\holim_{\sq^{n}_0}i_{n*}^+\sX).
\end{equation}

Let   $\rho^+_n:\sq^n\to \sq^{n+1}_0$ be the functor $\rho^+_n(I):=I\cup\{n+1\}$, giving the restriction functor
\[
\rho^+_{n*}:\sM^{\sq_0^{n+1}}\to \sM^{\sq^n}
\]
and the left adjoint $\rho_n^{+*}:\sM^{\sq^n}\to \sM^{\sq_0^{n+1}}$. Explicitly, 
for $\sX\in \sM^{\sq^n}$,  $\rho_n^{+*}\sX\in \sM^{\sq_0^{n+1}}$ is given  by $\rho_n^{+*}\sX(\rho^+_n(I))=\sX(I)$ and $\rho_n^{+*}\sX(J)=\pt$ for $J\subset \{1,\ldots, n\}$.

Similarly, we have a decomposition $\tilde{i}_n^\pm$ of $\sq^n$ into two $(n-1)$-cubes, with $\tilde{i}_n^-(I)=I$, 
$\tilde{i}_n^+(I)=I\cup\{n\}$ and natural transformation $\tilde\alpha_n:\tilde{i}_n^-\to \tilde{i}_n^+$, as for the punctured $(n+1)$-cube. We define the {\em iterated homotopy fiber} functor $\hofib^n:\sM^{\sq^n}\to \sM$ inductively by  
\[
\hofib^n(\sX)=\hofib\left(\hofib^{n-1}(\tilde{\alpha}_n):\hofib^{n-1}(\tilde{i}_n^-(\sX))\to \hofib^{n-1}(\tilde{i}_n^+(\sX))\right). 
\]
Using the isomorphism \eqref{eqn:SqInduction} and induction, we arrive at a natural isomorphism
\begin{equation}\label{eqn:IteratedHofib}
\hofib^n(\sX)\cong \holim_{\sq_0^{n+1}}\rho_n^{+*}(\sX).
\end{equation}

\begin{ex} \label{ex:ANCube} We let $\sM_0$ be one of the model categories discussed below in  \eqref{eqn:ModelCats}   and apply the above results to the stable model category $\sM$ of symmetric $T$-spectra $\Spt^{\Sigma,\sM_0}_T$, with $T=S^1$ or some model of $\P^1$. Let $\sE$ be a commutative monoid in $\Spt^{\Sigma,\sM_0}_T$. Form the cosimplicial (symmetric) spectrum $n\mapsto \sE^{\wedge n+1}$, with coface maps given by the appropriate unit maps and codegeneracies by  multiplication maps. Letting $\tilde\sE^{\wedge *+1}$ be a fibrant model,   lemma~\ref{lem:Truncation} and proposition~\ref{prop:SimpVCube} give us the isomorphisms in $\Ho\Spt^{\Sigma,\sM_0}_T\cong \Ho \Spt^{\sM_0}_T$
\begin{equation}\label{eqn:TotCubeIso}
\Tot_{(n)}\tilde\sE^{\wedge *+1}\cong \holim_{\Delta^{\le n}}\iota_{n*}\sE^{\wedge *+1}\cong \holim_{\sq^{n+1}_0}\phi^{n+1}_0\sE^{\wedge *+1},
\end{equation}
where we write $\phi^{n+1}_0\sE^{\wedge *+1}$ for $\phi^{n+1}_{0*}\iota_{n*}\sE^{\wedge *+1}$.

Let $\mS\in \Spt^{\Sigma,\sM_0}_T$ be the unit. We have the map $\mS\cong \Tot_nc\mS\to \Tot_n\tilde\sE^{\wedge *+1}$, induced by the unit map $c\mS\to \tilde\sE^{\wedge *+1}$. Letting $\bar\sE$ be the homotopy cofiber of the unit map $\mS\to \sE$, we claim there is a natural isomorphism in  $\Ho\Spt^{\Sigma,\sM_0}_T$
\[
\Omega^n\bar\sE^{\wedge n+1}\cong \hocofib(\mS\to \Tot_{(n)}\tilde\sE^{\wedge *+1}).
\]
Indeed, let $[\mS\to\sE]^{\wedge n+1}$ be the evident $(n+1)$-cube in spectra: $I\mapsto \sE^{\wedge |I|}$. The distinguished triangle $\mS\to \sE\to \bar\sE\to \mS[1]$ and isomorphism \eqref{eqn:IteratedHofib} give
the isomorphism
\[
\Omega^{n+1}\bar\sE^{\wedge n+1}\cong \holim_{\sq_0^{n+2}}\rho_{n+1}^{+*}[\mS\to\sE]^{\wedge n+1}
\]
in $\Ho\Spt^{\Sigma,\sM_0}_T$.  On the other hand, fill in the punctured $(n+1)$-cube $\phi^{n+1}_0\sE^{\wedge *+1}$ to an $(n+1)$-cube $\tilde\phi^{n+1}_0\sE^{\wedge *+1}$ by inserting $\pt$ at the entry $\0$, and similarly extend $\mS$ to an $(n+1)$-cube $\tilde\mS$ with value $\mS$ at $\0$ and value $\pt$ at $I\neq\0$. This gives us the homotopy fiber sequence in $(\Spt^{\Sigma,\sM_0}_T)^{\sq_0^{n+2}}$
\begin{equation}\label{eqn:HoFibSeq}
\rho_{n+1}^{+*}\tilde\phi^{n+1}_0\sE^{\wedge *+1}\to \rho_{n+1}^{+*}[\mS\to\sE]^{\wedge n+1}\to \rho_{n+1}^{+*}\tilde\mS.
\end{equation}

Using the isomorphism \eqref{eqn:IteratedHofib} gives us the isomorphisms (in $\Ho\Spt^{\Sigma,\sM_0}_T$)
\begin{align*}
&\Omega\holim_{\sq^{n+1}_0}\phi^{n+1}_0\sE^{\wedge *+1}\cong \holim_{\sq^{n+2}_0}\rho_{n+1}^{+*}\tilde\phi^{n+1}_0\sE^{\wedge *+1}\\
&\Omega^{n+1}\bar\sE^{\wedge n+1}\cong \holim_{\sq^{n+2}_0} \rho_{n+1}^{+*}[\mS\to\sE]^{\wedge n+1}\\
&\mS\cong \holim_{\sq^{n+2}_0}\tilde\mS.
\end{align*}
Thus, applying $\holim_{\sq^{n+2}_0}$ to the homotopy fiber sequence \eqref{eqn:HoFibSeq} gives us the distinguished triangle in $\Ho\Spt^{\Sigma,\sM_0}_T$
\[
\Omega\holim_{\sq^{n+1}_0}\phi^{n+1}_0\sE^{\wedge *+1}\to \Omega^{n+1}\bar\sE^{\wedge n+1}\to \mS\to\holim_{\sq^{n+1}_0}\phi^{n+1}_0\sE^{\wedge *+1},
\]
which  combined with \eqref{eqn:TotCubeIso} yields the desired result.
\end{ex}

We consider the case of $\sE=\MGL$ in $\Spt^\Sigma_T(S)$. For the construction of $\MGL$ we refer the reader to \cite{VoevICM}; for the structure as a symmetric monoidal object in $\Spt^\Sigma_T(S)$, we cite \cite[\S 2.1]{PPRMGL}. Applying the above example, we have the distinguished triangle in $\SH(S)$
\begin{equation}\label{eqn:CubeDistTriang}
 \mS_S\xrightarrow{i_n}\holim_{\sq^{n+1}_0}\phi^{n+1}_0\MGL^{\wedge *+1}\to  \Omega^n\overline{\MGL}^{\wedge n+1}\to\mS_S[1].
\end{equation}
Since $f_m^t$ is an exact functor and $\sq^{n+1}_0$ is a finite category, we have the isomorphism in $\SH(S)$
\[
\holim_{\sq^{n+1}_0}f_m^{t, \sq^{n+1}_0}\phi^{n+1}_0\MGL^{\wedge *+1}\cong f_m^t\holim_{\sq^{n+1}_0}\phi^{n+1}_0\MGL^{\wedge *+1}
\]
\begin{prop}\label{prop:SliceComp} 1. The morphism $i_n$ induces an isomorphism
\[
f^t_{m/N}\mS_S\to \holim_{\sq^{n+1}_0}f_{m/N}^{t, \sq^{n+1}_0}\phi^{n+1}_0\MGL^{\wedge *+1}
\]
for all $m\le N\le n+1$. \\
2. There is a natural isomorphism
\[
\xi_{m/N,n }:f^t_{m/N}\mS_S\to \Tot_{(n)}f^{t, \Delta}_{m/N}\MGL^{\wedge *+1}
\]
for $m\le N\le n+1$, compatible with the maps in the $\Tot_{(n)}$-tower (for fixed $m, N$ and varying $n$)  and the maps in the slice tower (for fixed $n$ and varying $m, N$).
\end{prop}

\begin{proof} The map $\mS_S\to \MGL$ induces an isomorphism $s^t_0\mS_S\to s^t_0\MGL$, and hence $s^t_0\overline{\MGL}=0$. As both $\mS_S$ and $\MGL$ are in $\SH^\eff(S)$, it follows that $f_1^t\overline{\MGL}=\overline{\MGL}$, and thus $f_{n+1}^t \Omega^n\overline{\MGL}^{\wedge n+1}\cong
 \Omega^n\overline{\MGL}^{\wedge n+1}$. From this follows 
 \[
 f_{m/N}^t \Omega^n\overline{\MGL}^{\wedge n+1}=0\text{ for }m\le N\le n+1.
 \]
Applying $f_{m/N}$ to the distinguished triangle \eqref{eqn:CubeDistTriang} completes the proof of (1).

For (2), the restriction $\iota_{n*}f^{t, \Delta}_{m/N}\MGL^{\wedge *+1}$ is fibrant in $\Spt^\Sigma_{\P^1}(S)^{\Delta^{\le n}}$ since by construction
 $f^{t, \Delta}_{m/N}\MGL^{\wedge *+1}$ is a fibrant object in $\Spt^\Sigma_{\P^1}(S)^\Delta$. In addition, we have an isomorphism in $\Ho\Spt^\Sigma_{\P^1}(S)^{\Delta^{\le n}}$
\[
\iota_{n*}f^{t, \Delta}_{m/N}\MGL^{\wedge *+1}\cong f^{t, \Delta^{\le n}}_{m/N} \iota_{n*}\MGL^{\wedge *+1}.
\]
Thus, by lemma~\ref{lem:Truncation},  we have a canonical isomorphism in $\SH(S)$
\[
\holim_{\Delta^{\le n}}f_{m/M}^{t, \Delta^{\le n}}\iota_{n*}\MGL^{\wedge *+1}\cong \Tot_{(n)}f^{t, \Delta}_{m/N}\MGL^{\wedge *+1}
\]
Similarly, by  proposition~\ref{prop:SimpVCube}, we have the isomorphism
\[
\holim_{\Delta^{\le n}}f_{m/M}^{t, \Delta^{\le n}}\iota_{n*}\MGL^{\wedge *+1}\cong \holim_{\sq^{n+1}_0}f_{m/M}^{t, \sq^{n+1}_0}\phi^{n+1}_0\MGL^{\wedge *+1}
\]
in $\SH(S)$; together with (1), these isomorphisms yield (2).
\end{proof}

\begin{rem} For the truncation functors $f_n^{\sS, t}$ and the Betti realization functor $\Re_B^\sS$, we have been using the projective model structure on the functor category, while for the $\Tot$-tower, we use the Reedy model structure. To pass from one situation to the other, we use lemma~\ref{lem:Truncation}  and proposition~\ref{prop:SimpVCube} to replace the $\Tot_{(n)}$-construction with a homotopy limit over the punctured $(n+1)$-cube. As  $\sq^{n+1}_0$ is a direct category, the Reedy model structure agrees with the projective model structure, so we may apply all these constructions freely. Besides the finiteness of  $\sq^{n+1}_0$, this is another reason why we pass from cosimplicial objects to cubes.
\end{rem}

\section{D\'ecalage}\label{sec:Decalage}  Deligne's d\'ecalage operation \cite[(1.3.3)]{HodgeII}  constructs a new filtration $\Dec F$ on a complex $K$  from a given filtration $F$ on  $K$; this change of filtration has the effect of accelerating the spectral sequence associated to the filtered complex $K$. Here we replace the filtered complex $K$ with a cosimplicial spectrum object together with the tower $\Tot^{(*)}$. The tower replacing $\Dec F$ turns out to arise from a suitable Postnikov tower, where the $n$th term is formed by applying the functor of the $(n-1)$-connected cover termwise to the given cosimplicial object and then applying $\Tot$. Our main result in this section is an analog of Deligne's comparison of the spectral sequences for $(K,F)$ and $(K, \Dec F)$ \cite[proposition 1.3.4]{HodgeII}.

For the application of this construction to the comparison of the slice and Adams-Novikov spectral sequence, we need only consider the model categories of simplicial sets and suspension spectra. However, with an eye to possible future applications, we will present this section in a somewhat more general setting. We were not able to formulate a good axiomatic description of the appropriate setting for this construction, rather, we give a list of examples, which we hope will cover enough ground to be useful.

We take $\sM_0$ to be one of the following pointed  closed symmetric monoidal simplicial model categories:
\begin{equation}\label{eqn:ModelCats}
\vbox{
\begin{enumerate}
\item $\Spc_\bullet$,  the category of pointed simplicial sets, with the usual model structure
\item Take $\sC$ to be a small category, $\tau$ a Grothendieck topology on $\sC$  and $\sM_0$ the category of $\Spc_\bullet$-valued presheaves on $\sC$ with the injective model structure (localized for the topology $\tau$). 
\item $B=\Spec k$, with $k$ a perfect field, $\sC=\Sm/B$, the category of smooth quasi-projective $B$-schemes and $\sM_0$ the category $\Spc_\bullet(B)$ with the motivic model structure, that is, the left Bousfield  localization of example (2) with $\sC=\Sm/B$, $\tau$ the Nisnevich topology, and the localization with respect to maps $\sX\wedge(\A^1,0)\to \pt$.  
\end{enumerate}}
\end{equation}
We note that these are all cofibrantly generated, cellular and combinatorial model categories. In case (2), we recall that the weak equivalences are given via the $\tau$-homotopy sheaves $\pi_n^\tau(\sX)$, this being the $\tau$-sheaf associated to the presheaf $U\mapsto [\Sigma^nU_+, \sX]_{\Ho\sM}$, and in case (3), the weak equivalences are given via the $\A^1$-homotopy sheaves $\pi_n^{\A^1}(\sX)$, these being similarly defined as the Nisnevich sheaf associated to the presheaf  $U\mapsto [\Sigma^nU_+, \sX]_{\Ho\sM}$.

For the stable model categories $\sM:=\Spt_T\sM_0$ we will   use the model structure induced from $\sM_0$ by the construction given in \cite[chapter 7]{Hovey}. We take in case (1) $T=S^1$, giving us the category of suspension spectra, with weak equivalences the stable weak equivalences. In (2), we take again the category of suspension spectra, where now $T=S^1$ acts through the simplicial structure. We assume that the weak equivalences are the stable weak equivalences, that is, maps that induce an isomorphism on the stable homotopy sheaves $\pi_n^s(\sE):=\colim_N\pi_{n+N}^\tau(\sE_N)$ if $\sE=(\sE_0, \sE_1,\ldots)$. 
In case (3), we may take $T=S^1$, giving the category of $S^1$-spectra $\Spt_{S^1}(B)$ or for the \'etale version $\Spt_{S^1}^\et(B)$. Here the weak equivalences are the stable weak equivalences, using the $\A^1$ homotopy sheaves $\pi_n^{\A^1}$ in place  of $\pi_n^\tau$. These are all cofibrantly generated, cellular, combinatorial stable simplicial $\sM_0$ model categories. If  at some point we require the stable category to have a symmetric monoidal model category structure, we will replace the spectrum category with symmetric spectra. 

In all cases, one has for $\sX$ homotopy objects $\pi_n(\sX)$, $n=0, 1,\ldots$, with $\pi_n$ an abelian group object for $n\ge2$, and a group object for $n=1$, so  that the $\{\pi_n, n\ge0\}$ detects weak equivalences, a loops functor $\sX\to \Omega\sX$ with $\pi_n(\Omega\sX)=\pi_{n+1}(\sX)$, so that a homotopy fiber sequence induces a long exact sequence in the $\pi_n$ in the usual extended sense, a functorial (left) Postnikov tower
\[
\ldots\to f_{n+1}\sX\to f_n\sX\to \ldots\to f_0\sX=\sX
\]
with $f_n\sX\to \sX$ inducing an isomorphism on $\pi_m$ for $m\ge n$ and with $\pi_m f_n\sX=\{*\}$ for $m<n$.  Furthermore,  for an integer $n\ge2$, there is an Eilenberg-MacLane space $K(A, n)$ associated to an abelian group (in case (1)) or $\tau$-sheaf of abelian groups (in case (2)) or strictly $\A^1$-invariant sheaf of abelian groups (in case (3)), which is determined up to unique isomorphism  in $\Ho\sM$ by the vanishing of $\pi_mK(A,n)$ for $m\neq n$ and the choice of an isomorphism $A\cong \pi_n K(A,n)$. 

For the spectrum categories, stabilizing the $\pi_n$ gives the collection of stable homotopy objects $\{\pi_n, n\in\ZZ\}$ which detect weak equivalences and which are abelian group objects for all $n$,  one has a functorial   (left) Postnikov tower
\[
\ldots\to f_{n+1}\sE\to f_n\sE\to\ldots\to  \sE
\]
and Eilenberg-MacLane spectrum $EM(A, n)$ for $A$ an abelian group object as above, and $n\in \ZZ$. 

In the sequel, we will treat all  these cases simultaneously; we will usually not need to distinguish between the stable and unstable setting, and will refer to the model category at hand as $\sM$, whether stable or unstable. We will retain the notation $K(A, n)$ for the Eilenberg-MacLane space in the unstable setting, and write $K(A, n)$ for the Eilenberg-MacLane spectrum $EM(A, n)$ in the stable case.

We apply the Postnikov tower construction in functor categories, as described in example~\ref{exs:FunctorialConstructions}(1), to  an object $\sX\in \sM^\Delta$, giving the cosimplicial object $f_n\sX\in \sM^\Delta$:
\[
f_n\sX:=[m\mapsto f_n\sX^m] 
\]
and the resulting tower
\[
\ldots\to f_{n+1}\sX\to f_n\sX\to\ldots\to \sX. 
\]
As the notation suggests, this tower has the property that evaluation at some $[m]\in\Delta$ yields the Postnikov tower for $\sX^m$. 

We will assume that we have a double delooping $\sY$ of $\sX$, that is, a weak equivalence  $\sX\to \Omega^2\sY$ in $\Ho\sM^\Delta$; we will simply replace $\sX$ with $\Omega^2\sY$, so we may assume that this weak equivalence is an identity. This assumption is of course fulfilled for all $\sX$ if we are in the stable case. We let $\sZ=\Omega\sY$ and use $\sZ$ as a chosen delooping of $\sX$.

\begin{Def}\label{Def:DecTowerMain} Fix an integer $A$ and an extended integer $B$, with $0\le A<B\le \infty$. Let $\sX$ be in $\sM^\Delta$. Applying the functor $\Tot^{(A/B)}$ to the Postnikov tower for $\sX$ gives the {\em  tower d\'ecal\'e} of spaces
 \begin{equation}\label{eqn:DecTower} 
\ldots\to \Tot^{(A/B)}(f_{n+1}\sX) \to \Tot^{(A/B)}(f_n\sX)\to\ldots\to  \Tot^{(A/B)}(\sX)
\end{equation}
\end{Def}

Using our chosen delooping $\sX= \Omega\sZ$, let $f_{k/m}\sX:=\hofib(f_{m+1}\sZ\to f_{k+1}\sZ)$.  Since $f_n\circ \Omega$ is naturally isomorphic to $\Omega\circ f_{n+1}$ as natural transformations to $\Ho\sM$,  the homotopy fiber sequence
\[
 \Omega f_{m+1}\sZ\to \Omega f_{k+1}\sZ\to  \hofib(f_{m+1}\sZ\to f_{k+1}\sZ)
\]
gives us the homotopy fiber sequence
\[
f_m\sX\to f_k\sX\to f_{k/m}\sX.
\]
We have as well induced delooping for $f_m\sX$, namely, $f_{m+1}\sZ$.

The tower \eqref{eqn:DecTower}  gives rise to the spectral sequence
\begin{equation}\label{eqn:DecTowerSS}
E_1^{p,q}(\Dec, \sX)=\pi_{-p-q}\Tot^{(A/B)} f_{(p/p+1)}\sX\Longrightarrow \pi_{-p-q}\Tot^{(A/B)}\sX
\end{equation}
for $0\le A<B\le \infty$.

The constructions $f_q$ and $\Tot^{(m/k)}$  are strictly functorial  and preserve homotopy fiber sequences.
Thus, we have the commutative diagram of natural transformations (for $0\le m< N\le\infty$, $0\le p$)
\[
\xymatrix{
\Tot^{(m+1/N)}(f_{p+1}(-))\ar[r]\ar[d]&\Tot^{(m+1/N)}(f_p(-))\ar[d]\\
\Tot^{(m/N)}(f_{p+1}(-))\ar[r]&\Tot^{(m/N)}(f_p(-))}
\]
and the homotopy fiber sequence
\[
\Tot^{({a}/{N})}(f_{p+1}(\sX))\to \Tot^{({a}/{N})}(f_p(\sX))\to \Tot^{({a}/{N})}(f_{{p}/{p+1}}(\sX)).
\]
for $0\le a\le N$.

The operation $f_{k/m}$  is also functorial and preserves fiber sequences, except that these are in terms of the chosen delooping $\sZ$ for $\sX$.  

Using the chosen delooping $\sZ$,  we define $\Tot^{(m/m+1)}_{p/p+1}(\sX)$ to be the homotopy fiber of the map 
\[
\Tot^{(m+1/N)}(f_{p+2}(\sZ))\to
\Tot^{(m/N)}(f_{p+1}(\sZ)).
\]
Note that  $\Tot^{(m/m+1)}_{p/p+1}(\sX)$ is, as the notation suggests, independent of the choice of $N$ (up to weak equivalence).

As  $\Tot^{(a/b)}$ commutes with $\Omega$, the homotopy fiber sequence
\[
\Omega\Tot^{(m+1/N)}(f_{p+2}(\sZ))\to
\Omega\Tot^{(m/N)}(f_{p+1}(\sZ))\to \Tot^{(m/m+1)}_{p/p+1}(\sX)
\]
yields the homotopy fiber sequence
\[
\Tot^{(m+1/N)}(f_{p+1}(\sX))\to
\Tot^{(m/N)}(f_{p}(\sX))\to \Tot^{(m/m+1)}_{p/p+1}(\sX).
\]

We have maps 
\begin{align*}
\&alpha:\Tot^{(m/m+1)}_{p/{p+1}}(\sX)\to \Tot^{(m/N)}(f_{p/p+1}(\sX)),\\
&\beta: \Tot^{(m/m+1)}_{p/{p+1}}(\sX)\to \Tot^{(m/m+1)}(f_p(\sX)),
\end{align*}
defined by taking the induced maps on the homotopy fibers of the horizontal maps   in the commutative diagrams
\begin{equation}\label{eqn:Diag1}
\xymatrix{
\Tot^{(m+1/N)}(f_{p+2}(\sZ))\ar[r]\ar[d]&\Tot^{(m/N)}(f_{p+1}(\sZ))\ar@{=}[d]\\
\Tot^{(m/N)}(f_{p+2}(\sZ))\ar[r]&\Tot^{(m/N)}(f_{p+1}(\sZ))}
\end{equation}

\begin{equation}\label{eqn:Diag2}
\xymatrix{
\Tot^{(m+1/N)}(f_{p+2}(\sZ))\ar[r]\ar[d]&\Tot^{(m/N)}(f_{p+1}(\sZ))\ar@{=}[d]\\
\Tot^{(m+1/N)}(f_{p+1}(\sZ))\ar[r]&\Tot^{(m/N)}(f_{p+1}(\sZ)).
}
\end{equation}
Via the homotopy fiber sequences
\begin{gather*}
\Tot^{(m/m+1)}(f_{p+1}(\sX))\to \Tot^{(m+1/N)}(f_{p+2}(\sZ))\to \Tot^{(m/N)}(f_{p+2}(\sZ))\\
\Tot^{(m+1/N)}(f_{p/p+1}(\sX))\to \Tot^{(m+1/N)}(f_{p+2}(\sZ))\to \Tot^{(m+1/N)}(f_{p+1}(\sZ)),
\end{gather*}
the Quetzalcoatl lemma gives the isomorphisms
\begin{equation}\label{eqn:HofFib}
\hofib(\alpha)\cong \Tot^{(m/m+1)}(f_{p+1}(\sX)),\ \hofib(\beta)\cong \Tot^{(m+1/N)}(f_{p/p+1}(\sX)).
\end{equation}
Putting all these together gives us the commutative diagram 
\begin{equation}\label{eqn:BigDiagr}
\xymatrixrowsep{12pt}
\xymatrixcolsep{8pt}
\xymatrix{
\Tot^{({m+1}/{N})}(f_{p+1}(\sX))\ar[r]\ar[d]\ar[rd]&\Tot^{({m+1}/{N})}(f_p(\sX))\ar[d]\ar[r]&\Tot^{({m+1}/{N})}(f_{{p}/{p+1}}(\sX))\ar[d]\\
\Tot^{({m}/{N})}(f_{p+1}(\sX))\ar[r]\ar[d]&\Tot^{({m}/{N})}(f_p(\sX))\ar[d]\ar[r]\ar[rd]&\Tot^{({m}/{N})}(f_{{p}/{p+1}}(\sX))\\
\Tot^{({m}/{m+1})}(f_{p+1}(\sX))\ar[r]&\Tot^{({m}/{m+1})}(f_p(\sX))&\Tot^{({m}/{m+1})}_{{p}/{p+1}}(\sX)\ar[l]^-\beta\ar[u]_\alpha
}
\end{equation}
with top two rows, the two left-hand columns and the diagonal all homotopy fiber sequences. 

\begin{lem}\label{lem:Iso} Let $p,q$ be integers with $p\ge0$, and $-2p \le q\le -p$. Take $N$ such that $2p+q+1\le N\le \infty$,  and consider the diagram extracted from \eqref{eqn:BigDiagr} with $m=2p+q$:
\[
\xymatrixrowsep{12pt}
\xymatrixcolsep{8pt}
\xymatrix{
&\pi_{-p-q}\Tot^{(m/{N})}(f_{{p}/{p+1}}(\sX))\\
\pi_{-p-q}\Tot^{({m}/{m+1})}(f_p(\sX))&\pi_{-p-q}\Tot^{({m}/{m+1})}_{{p}/{p+1}}(\sX)\ar[l]^-\beta\ar[u]_\alpha
}
\]
Then the map $\alpha$ is an isomorphism  and  the map $\beta$ is injective.
\end{lem}

\begin{proof} Let us first consider the map $\alpha$. The isomorphism \eqref{eqn:HofFib} gives us the homotopy fiber sequence
\begin{equation}\label{eqn:Seq1}
\Tot^{({m}/{m+1})}(f_{p+1}(\sX))\to\Tot^{({m}/{m+1})}_{{p}/{p+1}}(\sX)\xrightarrow{\alpha} \Tot^{({m}/{N})}(f_{{p}/{p+1}}(\sX)).
\end{equation}

Using the canonical isomorphism in $\Ho\sM^\Delta$,  $f_n(\Omega(\sT))\cong \Omega f_{n+1}(\sT)$ and the deloopings $\sX= \Omega\sZ$, $\sZ=\Omega\sY$, we may replace $\sX, \sZ$ with the pair $\sZ, \sY$ and identify the sequence \eqref{eqn:Seq1}  with $\Omega$ applied to the homotopy fiber sequence
\[
\Tot^{({m}/{m+1})}(f_{p+2}(\sZ))\to\Tot^{({m}/{m+1})}_{{p+1}/{p+2}}(\sZ)\xrightarrow{\alpha'} \Tot^{({m}/{N})}(f_{{p+1}/{p+2}}(\sZ)).
\]
From this, we see that  we may extend $\alpha$ to a  homotopy fiber sequence  
\[
\Tot^{({m}/{m+1})}_{{p}/{p+1}}(\sX)\xrightarrow{\alpha} \Tot^{({m}/{N})}(f_{{p}/{p+1}}(\sX))\to
\Tot^{({m}/{m+1})}(f_{p+2}(\sZ)).
\]
We have by lemma~\ref{lem:FiberSeq1}  
\[
\Tot^{({m}/{m+1})}(f_{p+2}(\sZ))\cong \Omega^{2p+q}Nf_{p+2}(\sZ^{2p+q}),
\]
hence $\pi_{-p-q+\epsilon}\Tot^{({m}/{m+1})}(f_{p+2}(\sZ))\cong N(\pi_{-p-q+\epsilon}(\Omega^{2p+q}f_{p+2}(\sZ^{2p+q})))$ is a subgroup of $\pi_{p+\epsilon}f_{p+2}(\sZ^{2p+q})$. Because
$\pi_{p+\epsilon}f_{p+2}(\sZ^{2p+q})$ is zero for $\epsilon=0,1$, $\alpha$ is an isomorphism.

For $\beta$, we have the homotopy fiber sequence
\[
\Tot^{({m+1}/{N})}(f_{{p}/{p+1}}(\sX))\to  \Tot^{({m}/{m+1})}_{{p}/{p+1}}(\sX)\xrightarrow{\beta} \Tot^{({m}/{m+1})}(f_{{p}/{p+1}}(\sX)).
\]

The cosimplicial object $f_{{p}/{p+1}}(\sX)$ is weakly equivalent to the cosimplicial Eilenberg-MacLane object
\[
n\mapsto K(\pi_{p}(\sX^n), p)
\]
hence $\pi_t\Tot^{({m+1}/{N})}(f_{{p}/{p+1}}(\sX))$ is the cohomology in degree $-t$ of the complex
\[
N\pi_{p}(\sX^{2p+q+1})\to N\pi_{p}(\sX^{2p+q+2})\to \ldots\to N\pi_{p}(\sX^{N-1}),
\]
concentrated in degrees $[p+q+1,N-p-1]$. So $\pi_{-p-q}\Tot^{({m+1}/{N})}(f_{{p}/{p+1}}(\sX))=0$ and $\beta$ is injective on $\pi_{-p-q}$.
\end{proof}

We consider the spectral sequences \eqref{eqn:TowerSS2} and \eqref{eqn:DecTowerSS} for $A=0$ and $0<B\le \infty$.  Take integers $p,q$ with $0\le -p$ and $0\le 2p+q<B$. We have
\[
E_1^{m,-p}(\sX)=N\pi_p\sX^m;
\]
the $E_1$-complex  $E_1^{*,-p}(\sX)$ is the (truncated) normalized complex (shifted to be supported in degrees $d$, $-p\le d\le B-p-1$)
\[
\sigma_{<B}N\pi_psX^*:=N\pi_p\sX^0\to \ldots\to N\pi_p\sX^{2p+q}\to N\pi_p\sX^{2p+q+1}\to \ldots\to N\pi_p\sX^{B-1}
\]
and $E_2^{2p+q,-p}=H^{p+q}(E_1^{*,-p}(\sX))$.

As $f_{p/p+1}(\sX)$ is weakly equivalent to the cosimplicial object
\[
m\mapsto K(\pi_p(\sX^m), p)
\]
it follows that $E_1^{p,q}(\Dec, \sX):=\pi_{-p-q}\Tot^{(0/B)}f_{p/p+1}(\sX)$ is $H^{p+q}$ of the complex (shifted to be supported in degrees $d$,  $-p\le d\le B-p-1$)
\[
\sigma_{<B}N\pi_p\sX^*:=N\pi_p\sX^0\to \ldots\to N\pi_p\sX^{2p+q}\to N\pi_p\sX^{2p+q+1}\to \ldots \to N\pi_p\sX^{B-1}.
\]
As this complex is equal to $E_1^{*,-p}(\sX)$, the identity maps on $N\pi_p\sX^*$ induce the isomorphism
\begin{equation}\label{eqn:DecE1}
\gamma_1^{p,q}:E_1^{p,q}(\Dec, \sX)\to E_2^{2p+q,-p}(\sX).
\end{equation}

\begin{prop} \label{prop:Dec}  Take $A=0$, $0<B\le \infty$. The maps \eqref{eqn:DecE1} give rise to an isomorphism of complexes
\[
\gamma_1^{*,q}:E_1^{*,q}(\Dec, \sX)\to E_2^{2*+q,-*}(\sX)
\]
and inductively to a sequence of isomorphisms  
\[
\gamma_r^{p,q}:E_r^{p,q}(\Dec, \sX)\to E_{r+1}^{2p+q,-p}(\sX)
\]
which give an isomorphism of complexes 
\[
\gamma_r^{*,*}:(\oplus_{p,q}E_r^{p,q}(\Dec, \sX), d_r)\to (\oplus_{p,q} E_{r+1}^{2p+q,-p}(\sX), d_{r+1})
\] 
for each $r\ge1$.
\end{prop}

\begin{proof} To simplify the notation, we give the proof in case $B=\infty$; the proof  in the general case is exactly the same, replacing $\Tot^{(-)}$ with $\Tot^{(-/B)}$ throughout. 

The spectral sequence \eqref{eqn:TowerSS2} is the spectral sequence associated to the exact couple
\[
\xymatrix{
D_1\ar[rr]^{i_1}&&D_1\ar[dl]^{\pi_1}\\
&E_1\ar[ul]^{\del_1}}
\]
with 
\[
D_1^{p,q}:=\pi_{-p-q}\Tot^{(p)}(\sX),\ E_1^{p,q}:=\pi_{-p-q}\Tot^{(p/p+1)}(\sX),
\]
the maps $i^{p,q}_1:D_1^{p+1,q-1}\to D_1^{p,q}$ and $\pi_1^{p,q}:D_1^{p,q}\to E_1^{p,q}$  induced by the canonical morphisms
\begin{align*}
&\Tot^{(p+1)}(\sX)\to \Tot^{(p)}(\sX),\\
&\Tot^{(p)}(\sX)\to \Tot^{(p/p+1)}(\sX),
\end{align*}
respectively,  and with  $\del_1^{p,q}:E_1^{p,q}\to D_1^{p+1,q}$ the boundary map associated to the homotopy fiber sequence
\[
\Tot^{(p+1)}(\sX)\to\Tot^{(p)}(\sX)\to \Tot^{(p/p+1)}(\sX).
\]
Similarly, the spectral sequence \eqref{eqn:DecTowerSS} arises from the exact couple
\[
\xymatrix{
D_{1,\Dec}\ar[rr]^i&&D_{1,\Dec}\ar[dl]^\pi\\
&E_{1,\Dec}\ar[ul]^\del}
\]
defined in a similar way, where we replace $\Tot^{(p)}(\sX)$, $\Tot^{(p+1)}(\sX)$ and $\Tot^{(p/p+1)}(\sX)$ with
$\Tot^{(0)} f_p(\sX)$, $\Tot^{(0)} f_{p+1}(\sX)$ and $\Tot^{(0)} f_{p/p+1}(\sX)$. To prove the result, it suffices to define maps
\[
\delta^{p,q}_1: D_{1,\Dec}^{p,q}\to D_2^{2p+q,-p}
\]
such that 
\[
\lower20pt\hbox{$\begin{pmatrix}\delta_1&&\delta_1\\&\gamma_1\end{pmatrix}:$}
\xymatrix{
D_{1,\Dec}\ar[rr]^{i_1}&&D_{1,\Dec}\ar[dl]^{\pi_1}\\
&E_{1,\Dec}\ar[ul]^{\del_1}}
\lower20pt\hbox{$\to$} 
\xymatrix{
D_2\ar[rr]^{i_2}&&D_2\ar[dl]^{\pi_2}\\
&E_2\ar[ul]^{\del_2}}
\]
defines a map of (reindexed) exact couples.

We recall that $E_2$ is the cohomology of the complex $(E_1, d_1)$, with $d_1=\pi_1\circ \del_1$. Let $Z_2\subset E_1$ be the kernel  of $d_1$ and note that $Z_2\supset\pi_1(D_1)$. By definition,   $D^{p,q}_2=i_1(D^{p,q}_1)\subset D_1^{p-1,q+1}$, $i_2:D_2\to D_2$ is the map induced by $i_1$,  the map $\pi_2:D_2\to E_2$ is  defined by the commutative diagram
\[
\xymatrix{
D_2\ar@{^(->}[d]\ar[r]_{\pi_{1|D_2}}\ar@/^0.5cm/[rr]^{\pi_2}&Z_2\ar@{->>}[r]_\pi\ar@{^(->}[d]&E_2\\
D_1\ar[r]_{\pi_1}&E_1}
\]
and $\del_2:E_2\to D_2$ is induced by restricting $\del_1$ to $Z_2$, noting that this restriction sends $Z_2$ to $i_1(D_1)\subset D_1$, and descends to $E_2$.

Next, we note that the maps
\begin{align*}
&\pi_{-p-q}\Tot^{(2p+q)}f_p\sX\to \pi_{-p-q}\Tot^{(0)}f_p\sX\\
&\pi_{-p-q}\Tot^{(2p+q)}f_{p/p+1}\sX\to \pi_{-p-q}\Tot^{(0)} f_{p/p+1}\sX\\
&\pi_{-p-q-1}\Tot^{(2p+q+2)}f_{p+1}\sX\to \pi_{-p-q-1}\Tot^{(0)} f_{p+1}\sX
\end{align*}
are surjective and
\begin{align*}
&\pi_{-p-q}\Tot^{(2p+q-1)}f_p\sX\to \pi_{-p-q}\Tot^{(0)}f_p\sX\\
&\pi_{-p-q}\Tot^{(2p+q-1)}f_{p/p+1}\sX\to \pi_{-p-q}\Tot^{(0)} f_{p/p+1}\sX\\
&\pi_{-p-q-1}\Tot^{(2p+q+\epsilon)}f_{p+1}\sX\to \pi_{-p-q-1}\Tot^{(0)} f_{p+1}\sX
\end{align*}
($\epsilon\le1$) are isomorphisms, by lemma~\ref{lem:conn}.  From the commutative diagram
\[
\xymatrix{
\pi_{-p-q}\Tot^{(2p+q)}f_p\sX\ar@{->>}[r]\ar[d]&
\pi_{-p-q}\Tot^{(2p+q-1)}f_p\sX\ar[r]^-\sim\ar[d]& \pi_{-p-q}\Tot^{(0)} f_p\sX\\
\pi_{-p-q}\Tot^{(2p+q)}\sX\ar[r]&\pi_{-p-q}\Tot^{(2p+q-1)}\sX}
\]
we arrive at the well-defined map
\begin{multline*}
D_{1,\Dec}^{p,q}=\pi_{-p-q}\Tot^{(0)} f_p\sX
\xrightarrow{\delta_1^{p,q}}\im[\pi_{-p-q}\Tot^{(2p+q)}\sX\to \pi_{-p-q}\Tot^{(2p+q-1)}\sX]\\= D_2^{2p+q,-p}.
\end{multline*}
The identity
\[
i_2\circ\delta_1=\delta_1\circ i_{1,\Dec}
\]
follows directly.

To show that $\pi_2\circ \delta_1=\gamma_1\circ \pi_{1,\Dec}$, we consider the diagram (which is well-defined by lemma~\ref{lem:Iso})
\begin{equation}\label{eqn:BigDiagr2}
\xymatrix{
\ar@/_2.5cm/[dddd]_{\delta_1^{p,q}}\pi_{-p-q}\Tot^{(0)} f_p\sX\ar[r]^{\pi_{1,\Dec}}&\pi_{-p-q}\Tot^{(0)} f_{p/p+1}\sX\ar@/^2cm/[ddr]^{\gamma_1^{p,q}}\\
\pi_{-p-q}\Tot^{(2p+q)}f_p\sX\ar[r]\ar[dd]\ar[dr]\ar@{->>}[u]&\pi_{-p-q}\Tot^{(2p+q)}f_{p/p+1}\sX\ar@{_(->}[d]^{\beta\circ \alpha^{-1}}\ar@{->>}[u]
\ar@/^0.5cm/[ddr]^{\tilde\gamma_1^{p,q}}\\\
&\pi_{-p-q}\Tot^{(2p+q/2p+q+1)}f_p\sX\ar[d]&E_2^{p,q}\\
\pi_{-p-q}\Tot^{(2p+q)}\sX\ar[r]_-{\pi_1}&\pi_{-p-q}\Tot^{(2p+q/2p+q+1)}\sX&Z_2^{2p+q,-p}\ar@{_(->}[l]\ar@{->>}[u]_\pi\\
D_2^{2p+q,-p}\ar@{_(->}[u]\ar@/_1cm/[rru]_{\pi_{1|D_2}}}
\end{equation}
The right-hand column may be described explicitly as follows: let 
\[
N\pi_p\sX^*:=[N\pi_p\sX^0\to N\pi_p\sX^1\to\ldots]
\]
be the  normalized complex associated to the cosimplicial abelian group object $n\mapsto \pi_p\sX^n$, shifted to be supported in degrees $[-p,\infty)$. Then the right-hand column is the sequence of evident maps
\[
\xymatrix{
H^{p+q}(N\pi_p\sX^*)&
Z^{p+q}(N\pi_p\sX^*)\ar@{->>}[l]\ar@{^(->}[r]&
N\pi_p\sX^{2p+q}\ar@{=}[r]&
N\pi_p\sX^{2p+q}}
\]
The map $\tilde\gamma_1^{p,q}$ is the evident identification of $Z^{p+q}(N\pi_p\sX^*)$ with $Z_2^{2p+q,-p}$. The commutativity of \eqref{eqn:BigDiagr2} follows from this computation and the commutativity of diagram \eqref{eqn:BigDiagr}. Since $\pi_2=\pi\circ \pi_{1|D_2}$, this shows that 
$\pi_2\circ \delta_1=\gamma_1\circ \pi_{1,\Dec}$.

For the remaining identity $\del_2\circ \gamma_1=\delta_1\circ\del_1$,  we extract from the diagram \eqref{eqn:BigDiagr} a commutative diagram (in $\sM$)  with rows being homotopy fiber sequences and $m=2p+q$
\[
\xymatrix{
\Tot^{({m+1})}f_p\sX\ar[r]&\Tot^{({m})}f_p\sX\ar[r]&\Tot^{({m}/{m+1})}f_p\sX\\
\Tot^{({m+1})}f_{p+1}\sX\ar[r]\ar[d]\ar[u]&\Tot^{({m})}f_p\sX\ar@{=}[d]\ar[r]\ar@{=}[u]&\Tot^{({m}/{m+1})}_{{p}/{p+1}}\sX\ar[d]_\alpha\ar[u]_\beta\\
\Tot^{({m})}f_{p+1}\sX\ar[r]&\Tot^{({m})}f_p\sX\ar[r]&\Tot^{({m})}f_{{p}/{p+1}}\sX.
}
\]
This gives us the commutative diagram
\[
\xymatrixcolsep{30pt}
\xymatrix{
\pi_{-p-q}\Tot^{({m}/{m+1})}\sX\ar[r]^\del&\pi_{-p-q-1}\Tot^{({m+1})}\sX\\
\pi_{-p-q}\Tot^{({m}/{m+1})}f_p\sX\ar[r]^{\del^{{m}/{m+1}}}\ar[u]^{\tilde\gamma}
&\pi_{-p-q-1}\Tot^{({m+1})}f_p\sX\ar[u]_-{\tilde\delta}\\
\pi_{-p-q}\Tot^{({m}/{m+1})}_{{p}/{p+1}}\sX\ar[d]^\sim_\alpha\ar@{_(->}[u]^\beta\ar[r]^-{\tilde\del}&\pi_{-p-q-1}\Tot^{({m+1})}f_{p+1}\sX\ar[d]_\sim^{\tilde\alpha}\ar[u]_j\\
\pi_{-p-q}\Tot^{({m})}f_{{p}/{p+1}}\sX\ar[r]_{\del_{p/p+1}}\ar[d]^\sim_\phi&\pi_{-p-q-1}\Tot^{({m})}f_{p+1}\sX\ar[d]_\sim^{\tilde\phi}\\
\pi_{-p-q}\Tot^{({0})} f_{{p}/{p+1}}\sX\ar[r]_-{\del_{1,\Dec}}&\pi_{-p-q-1}\Tot^{({0})} f_{p+1}\sX.
}
\]
The map $\del_2$ is induced from $\del$, the map $\delta_1$ is induced from $\tilde\delta\circ j\circ\tilde{\alpha}^{-1}\circ\tilde{\phi}^{-1}$ (noting that this latter map has image in $D_2^{2p+q+2,-p-1}$), and $\gamma_1=\tilde\gamma\circ\beta\circ\alpha^{-1}\circ\phi^{-1}$ (as we have noted above). This gives the identity $\del_2\circ \gamma_1=\delta_1\circ\del_{1,\Dec}$, completing the proof.
\end{proof}

\begin{rem} Proposition~\ref{prop:Dec} may be viewed as a homotopy theoretic analog of a special case of Deligne's result \cite[proposition 1.3.4]{HodgeII}. Indeed, let $K^{**}$ be a double complex and let $K^*$ be the associated (extended) total complex
\[
K^n:=\prod_{a+b=n}K^{a,b}.
\]
Give $K^n$ the filtration by taking the stupid filtration in the first variable, that is, $(F^mK)^n:=\prod_{a+b=n, a\ge m}K^{a,b}$. Then Deligne's filtration $\Dec^m K^*$ is given by  $\Dec^m K^n=\prod_{a+b=n}\Dec^mK^{a,b}$ with
\[
\Dec^mK^{a,b}=\begin{cases} K^{a,b}&\text{ for } b<-m\\0&\text{ for } b>-m\\\ker(\del_2:K^{a, -m}\to K^{a,-m+1})&\text{ for } b=-m.\end{cases}
\]
That is, $\Dec^m K^*$ is the extended total complex of the double complex 
\[
a\mapsto \tau^{can}_{\le -m}(K^{a,*},\del_2), 
\]
$\tau^{can}_{\le -m}C^*$ being the canonical subcomplex of a complex $C^*$.

If  $K^{a,b}=0$ for $a<0$, we may use the Dold-Kan correspondence to give a cosimplicial object in complexes
\[
n\mapsto \tilde K^{n, *}
\]
such that $K^{a, *}=N\tilde K^{a,*}$ as complexes, and the differential $\del_1:K^{a,b}\to K^{a+1, b}$ is the differential $N\tilde K^{a,*}\to N\tilde K^{a+1,*}$ given as the usual alternating sum of coface maps. If we let $EM \tilde{K}^{a,*}$ be the Eilenberg-MacLane spectrum associated to the complex $\tilde{K}^{a,*}$, then $\Tot[n\mapsto EM\tilde{K}^{n,*}]$ is the Eilenberg-MacLane spectrum associated to $\Tot\,K^{*}$, the tower $\Tot^{(*)}[n\mapsto EM\tilde{K}^{n,*}]$ is the tower associated to the filtration $F^*K^*$, and the tower $\Tot\,[n\mapsto f_*EM\tilde{K}^{n,*}]$ is associated to $\Dec^*K$.  Furthermore, the spectral sequences \eqref{eqn:TowerSS2}  and \eqref{eqn:DecTowerSS} are the same as the ones associated to the filtered complex $F^*K$ and $\Dec^*K$, respectively,  and the isomorphism of Proposition~\ref{prop:Dec} is the same as that of \cite[proposition 1.3.4]{HodgeII}. The proof given here is considerably more involved than that in \cite{HodgeII}, due to the fact that one could not simply compute with elements, as is possible in the setting of filtered complexes.
\end{rem}

\end{document}